\documentclass[12pt,british]{amsart}
\usepackage[T1]{fontenc}
\usepackage[latin9]{inputenc}
\usepackage{geometry}
\usepackage{babel}
\usepackage{amsbsy}
\usepackage{amstext}
\usepackage{amsthm}
\usepackage{amssymb}
\usepackage{stackrel}
\usepackage{setspace}
\usepackage[unicode=true,pdfusetitle,
 bookmarks=true,bookmarksnumbered=false,bookmarksopen=false,
 breaklinks=false,pdfborder={0 0 0},pdfborderstyle={},backref=false,colorlinks=false]
 {hyperref}

\makeatletter

\newcommand{\noun}[1]{\textsc{#1}}

\numberwithin{equation}{section}
\numberwithin{figure}{section}
\theoremstyle{plain}
\newtheorem{thm}{\protect\theoremname}[section]
\theoremstyle{remark}
\newtheorem{rem}[thm]{\protect\remarkname}
\theoremstyle{plain}
\newtheorem{cor}[thm]{\protect\corollaryname}
\theoremstyle{definition}
\newtheorem{defn}[thm]{\protect\definitionname}
\theoremstyle{plain}
\newtheorem{lem}[thm]{\protect\lemmaname}
\theoremstyle{plain}
\newtheorem{prop}[thm]{\protect\propositionname}
\theoremstyle{definition}
\newtheorem{example}[thm]{\protect\examplename}

\makeatother

\providecommand{\corollaryname}{Corollary}
\providecommand{\definitionname}{Definition}
\providecommand{\examplename}{Example}
\providecommand{\lemmaname}{Lemma}
\providecommand{\propositionname}{Proposition}
\providecommand{\remarkname}{Remark}
\providecommand{\theoremname}{Theorem}

\begin{document}
\title{Jordan-Lie inner ideals of finite dimensional associative algebras}
\author{Alexander Baranov}
\thanks{Supported by the University of Leicester}
\address{Department of Mathematics, University of Leicester, Leicester, LE1
7RH, UK}
\email{ab155@le.ac.uk }
\author{Hasan M. Shlaka}
\thanks{Supported by the Higher Committee for Educational Development in Iraq
(HCED Iraq). }
\address{Department of Mathematics, University of Leicester, Leicester, LE1
7RH, UK; Department of Chemical Engineering, University of Kufa, Al-Najaf,
Iraq.}
\email{hasan.shlaka@uokufa.edu.iq}
\begin{abstract}
We study Jordan-Lie inner ideals of finite dimensional associative
algebras and the corresponding Lie algebras and prove that they admit
Levi decompositions. Moreover, we classify Jordan-Lie inner ideals
satisfying a certain minimality condition and show that they are generated
by pairs of idempotents.
\end{abstract}

\keywords{Jordan-Lie inner ideals, Levi decomposition, idempotents, Lie structure
of an associative algebra.}
\subjclass[2000]{17B60}

\maketitle
\global\long\def\bbR{\mathbb{R}}%

\global\long\def\ccR{\mathcal{R}}%

\begin{singlespace}
\global\long\def\bbF{\mathbb{F}}%

\end{singlespace}

\global\long\def\ccF{\mathcal{F}}%

\global\long\def\ccL{\mathcal{L}}%

\global\long\def\ccP{\mathcal{P}}%

\global\long\def\ccB{\mathcal{B}}%

\global\long\def\ccM{\mathcal{M}}%

\begin{singlespace}
\global\long\def\dlim{\operatorname{\underrightarrow{{\rm lim}}}}%

\global\long\def\ker{\operatorname{\rm ker}}%

\global\long\def\Im{\operatorname{\rm Im}}%

\global\long\def\End{\operatorname{\rm End}}%

\global\long\def\dim{\operatorname{\rm dim}}%

\global\long\def\core{\operatorname{\rm core}}%
 
\end{singlespace}

\global\long\def\skew{\operatorname{\rm skew}}%

\global\long\def\Soc{\operatorname{\rm Soc}}%

\global\long\def\Range{\operatorname{\rm Range}}%

\global\long\def\rank{\operatorname{\rm rank}}%

\global\long\def\ad{\operatorname{\rm ad}}%

\global\long\def\gl{\operatorname{\rm \mathfrak{gl}}}%

\global\long\def\sl{\operatorname{\rm \mathfrak{sl}}}%

\global\long\def\sp{\operatorname{\rm \mathfrak{sp}}}%

\global\long\def\so{\operatorname{\rm \mathfrak{so}}}%

\global\long\def\su{\operatorname{\rm \mathfrak{su^{*}}}}%

\global\long\def\u{\operatorname{\rm \mathfrak{u^{*}}}}%

\global\long\def\rank{\operatorname{\rm rk}}%

\global\long\def\rad{\operatorname{\rm rad}}%

\section{Introduction}

Let $L$ be a Lie algebra. A subspace $B$ of $L$ is said to be an
\emph{inner ideal} of $L$ if $[B,[B,L]]\subseteq B$. Note that every
ideal is an inner ideal. On the other hand, there are inner ideals
which are not even subalgebras. This makes them notoriously difficult
to study. Inner ideals of Lie algebras were first introduced by Benkart
\cite{Benkart-1,Benkart}. She showed that inner ideals and $\ad$-nilpotent
elements of Lie algebras are closely related. Since certain restrictions
on the $\ad$-nilpotent elements yield an elementary criterion for
distinguishing the non-classical from classical simple Lie algebras
in positive characteristic, inner ideals play a fundamental role in
classifying Lie algebras \cite{Premet,Premet1}. Inner ideals are
useful in constructing grading for Lie algebras \cite{Lop=000026Gar=000026Loz=000026Neh}.
It was shown in \cite{Lop=000026Gar=000026Loz} that inner ideals
play role similar to that of one-sided ideals in associative algebras
and can be used to develop Artinian structure theory for Lie algebras.
Inner ideals of classical Lie algebras were classified by Benkart
and Fernández López \cite{Benkart-1,Ben=000026Lop}, using the fact
that these algebras can be obtained as the derived Lie subalgebras
of (involution) simple Artinian associative rings. In this paper we
use a similar approach to study inner ideals of the derived Lie subalgebras
of finite dimensional associative algebras. These algebras generalize
the class of simple Lie algebras of classical type and are closely
related to the so-called root-graded Lie algebras \cite{Baranov1}.
They are also important in developing representation theory of non-semisimple
Lie algebras (see \cite{Bav=000026Zal,Bav=000026Zal-1}). As we do
not require our algebras to be semisimple we have a lot more inner
ideals to take care of (including all ideals!), so some reasonable
restrictions are needed. We believe that such a restriction is the
notion of a Jordan-Lie inner ideal introduced by Fernández López in
\cite{Lopez}. We need some notation to state our main results. 

The ground field $\bbF$ is algebraically closed of characteristic
$p\ge0$. Let $A$ be a finite dimensional associative algebra over
$\bbF$ and let $R$ be the radical of $A$. Recall that $A$ becomes
a Lie algebra $A^{(-)}$ under $[x,y]=xy-yx$. Put $A^{(0)}=A^{(-)}$
and $A^{(k)}=[A^{(k-1)},A^{(k-1)}]$, $k\geq1$. Let $L=A^{(k)}$
for some $k\ge0$ and let $B$ be an inner ideal of $L$. Then $B$
is said to be \emph{Jordan-Lie} if $B^{2}=0$ (in that case, $B$
is also an inner ideal of the Jordan algebra $B^{+}$\cite{Lopez}).
Denote by $\bar{B}$ the image of $B$ in $\bar{L}=L/R\cap L$. Let
$X$ be an inner ideal of $\bar{L}$. We say that $B$ is \emph{$X$-minimal}
(or simply, \emph{bar-minimal}) if $\bar{B}=X$ and for every inner
ideal $B'$ of $L$ with $\bar{B}'=X$ and $B'\subseteq B$ we have
$B'=B$. Let $e$ and $f$ be idempotents in $A$. Then the\emph{
idempotent pair} $(e,f)$ is said to be \emph{orthogonal }if $ef=fe=0$
and \emph{strict }if for each simple component \emph{$S$ }of $\bar{A}=A/R$,
the projections of $\bar{e}$ and $\bar{f}$ on $S$ are either both
zero or both non-zero. We are now ready to state our main results.
\begin{thm}
\label{thm:main iff} Let $A$ be a finite dimensional associative
algebra (over an algebraically closed field of characteristic $p\ge0$)
and let $B$ be a Jordan-Lie inner ideal of $L=A^{(k)}$ ($k\ge0$).
Suppose $p\neq2,3$. Then $B$ is bar-minimal if and only if $B=eAf$
where $(e,f)$ is a strict orthogonal idempotent pair in $A$.
\end{thm}

Let $e$, $e'$, $f$ and $f'$ be idempotents in $A$. We write $(e,f)\overset{\ccL\ccR}{\le}(e',f')$
if $e'e=e$ and $ff'=f$; $(e,f)\le(e',f')$ if $e'e=ee'=e$ and $ff'=f'f=f$;
$(e,f)\overset{\ccL\ccR}{\sim}(e',f')$ if $(e,f)\overset{\ccL\ccR}{\le}(e',f')$
and $(e',f')\overset{\ccL\ccR}{\le}(e,f)$ (see Definition \ref{def:idem pair}).
Let $E$ be the set of all strict orthogonal idempotent pairs in $A$.
Then $\overset{\ccL\ccR}{\le}$ is a preorder on $E$ and it induces
a partial order on the quotient set $E/\overset{\ccL\ccR}{\sim}$,
see Remark \ref{rem:LRpreorder}(1). The following result is used
to describe the poset of bar-minimal Jordan-Lie inner ideals of $A$
and has an independent interest.
\begin{thm}
\label{thm:eAf sub e'Af' imply ...} Let $A$ be an Artinian ring
or a finite dimensional associative algebra over any field and let
$(e,f)$ and $(e',f')$ be idempotent pairs in $A$. Suppose that
$(e,f)$ is strict. Then the following hold.
\end{thm}

\begin{enumerate}
\item[(i)]  If $(e,f)\ne(0,0)$ then $eAf\ne0$.
\item[(ii)]  $eAf\subseteq e'Af'$ if and only if $(e,f)\overset{\ccL\ccR}{\le}(e',f')$.
\item[(iii)]  Suppose that $(e',f')$ is strict. Then $eAf=e'Af'$ if and only
if $(e,f)\overset{\ccL\ccR}{\sim}(e',f')$.
\item[(iv)]  The map $\xi:(e,f)\mapsto eAf$ induces a bijection from the poset
$(E/\overset{\ccL\ccR}{\sim},\overset{\ccL\ccR}{\le})$ to the poset
$(\xi(E),\subseteq)$.
\item[(v)]  Suppose that $eAf\subseteq e'Af'$. Then there exists a strict idempotent
pair $(e'',f'')$ in $A$ such that $(e'',f'')\le(e',f')$, $(e'',f'')\overset{\ccL\ccR}{\sim}(e,f)$
and $e''Af''=eAf$.
\end{enumerate}
\begin{rem}
It is well-known that every finite dimensional unital algebra is Artinian
as a ring. In particular, semisimple finite dimensional algebras are
Artinian. However, this is not true for non-unital algebras (e.g.
for the one dimensional algebra over $\mathbb{Q}$ with zero multiplication).
This is why we refer to both Artinian rings and finite dimensional
algebras in the theorem above.
\end{rem}

\begin{cor}
\label{cor:innerposet} Let $A$ be a finite dimensional associative
algebra and let $L=A^{(k)}$ ($k\ge0$). Suppose $p\neq2,3$. Then
the map $\xi:(e,f)\mapsto eAf$ induces an isomorphism from the poset
$(E/\overset{\ccL\ccR}{\sim},\overset{\ccL\ccR}{\le})$ to the poset
of all bar-minimal Jordan-Lie inner ideals of $L$.
\end{cor}

Let $B$ be an inner ideal of $L=A^{(k)}$ ($k\ge0$). Then $B$ is
said to be \emph{regular} (with respect to $A$) if $B$ is Jordan-Lie
(i.e. $B^{2}=0$) and $BAB\subseteq B$ (see also Proposition \ref{prop:Bav=000026Row 4.9}
for an alternative description in terms of the orthogonal pairs of
one-sided ideals of $A$). 
\begin{cor}
\label{cor:main regular} Let $A$ be a finite dimensional associative
algebra and let $L=A^{(k)}$ ($k\ge0$). Let $B$ be a Jordan-Lie
inner ideal of $L$. Suppose $p\neq2,3$ and $B$ is bar-minimal.
Then $B$ is regular. 
\end{cor}

It follows that all bar-minimal inner ideals are regular. It was also
proved in \cite[4.11]{Bav=000026Lop} that all maximal abelian inner
ideals of simple rings are regular. Note that any subspace $B$ of
$A$ satisfying the regularity conditions $B^{2}=0$ and $BAB\subseteq B$
is an inner ideal, i.e. $[B,[B,A]]\subseteq B$ holds. The first two
conditions are much easier to check, so it is important to know when
all Jordan-Lie inner ideals of $A^{(k)}$ are regular. We believe
that this holds for most finite dimensional algebras $A$. However,
exceptions do exists, as Examples \ref{nr-1} and \ref{nr} show.

Recall that every finite dimensional associative algebra $A$ over
an algebraically closed field admits Wedderburn-Malcev decomposition
$A=S\oplus R$ where $R$ is the radical of $A$ and $S$ is a maximal
semisimple subalgebra of $A$. A similar result (called Levi decomposition)
exist for Lie algebras. Following the Lie algebras terminology, we
refer to $S$ as a \emph{Levi subalgebra} of $A$. Let $B$ be an
inner ideal of $L=A^{(k)}$ ($k\ge0$). Then we say that $B$ \emph{splits}
in $A$ if there is a Levi subalgebra $S$ of $A$ such that $B=B_{S}\oplus B_{R}$,
where $B_{S}=B\cap S$ and $B_{R}=B\cap R$ (Definition \ref{def:splitness}).
\begin{cor}
\label{cor:main split} Let $A$ be a finite dimensional associative
algebra and let $L=A^{(k)}$ ($k\ge0$). Let $B$ be a Jordan-Lie
inner ideal of $L$. Suppose $p\neq2,3$. Then $B$ splits in $A$.
More exactly, there is a Levi subalgebra $S$ of $A$ and a strict
orthogonal idempotent pair $(e,f)$ in $S$ such that $B=eSf\oplus B_{R}$
with $eRf\subseteq B_{R}=B\cap R$.
\end{cor}

\section{Preliminaries}

Throughout this paper, unless otherwise specified, the ground field
$\bbF$ is algebraically closed of characteristic $p\ge0$; $A$ is
a finite dimensional associative algebra over $\bbF$; $R=\rad A$
is the radical of $A$; $S$ is a Levi (i.e. maximal semisimple) subalgebra
of $A$, so $A=S\oplus R$; $L=A^{(k)}$ for some $k\ge0$; $\rad L$
is the solvable radical of $L$ and $N=R\cap L$ is the \emph{nil-radical}
of $L$. If $V$ is a subspace of $A$, we denote by $\bar{V}$ its
image in $\bar{A}=A/R$. In particular, $\bar{L}=(L+R)/R\cong L/N$.
Since $R$ is a nilpotent ideal of $A$ the ideal $N=R\cap L$ of
$L$ is also nilpotent, so $N\subseteq\rad L$. It is easy to see
that $N=\rad L$ if $p=0$ and $k\ge1$, so $L/N$ is semisimple in
that case. Recall that a Lie algebra $L$ is \emph{perfect} if $[L,L]=L$.
We denote by $\ccM_{n}$ the algebra of $n\times n$ matrices over
$\bbF$ and by $\sl_{n}$ the Lie subalgebra of $\ccM_{n}$ consisting
of zero trace matrices, so $\sl_{n}=\ccM_{n}^{(1)}$.
\begin{defn}
Let $Q$ be a Lie algebra. We say that $Q$ is a \emph{quasi (semi)simple}
if $Q$ is perfect and $Q/Z(Q)$ is (semi)simple. 
\end{defn}

Herstein \cite[Theorem 4]{Herstein} proved that if $A$ is a simple
ring of characteristic different from $2$, then $A^{(1)}=[A,A]$
is a quasi simple Lie ring. As a special case, we note the following
well-known fact. 
\begin{lem}
\label{lem:sl_n} Let $p\ne2$, $n\ge2$ and let $A=\ccM_{n}$. Then
$[A,A]=\sl_{n}$ is quasi simple. In particular, $A^{(\infty)}=A^{(1)}$. 
\end{lem}

Note that the case of $p=2$ is exceptional indeed as the algebra
$\sl_{2}$ is solvable in characteristic $2$. 
\begin{prop}
\label{prop:=00005BA,A=00005D quasi Levi} Suppose $A$ is semisimple
and $p\neq2$. Then $[A,A]$ is quasi semisimple. In particular, $A^{(\infty)}=A^{(1)}$. 
\end{prop}

\begin{proof}
Since $A$ is semisimple, $A=\bigoplus_{i\in I}S_{i}$ where the $S_{i}$
are simple ideals of $A$. Since $\bbF$ is algebraically closed,
$S_{i}\cong\ccM_{n_{i}}$ for some $n_{i}$. Note that $[S_{i},S_{i}]=0$
if $n_{i}=1$ and $[S_{i},S_{i}]\cong\sl_{n_{i}}$ if $n_{i}\ge2$.
Now the result follows from Lemma \ref{lem:sl_n}.
\end{proof}
\begin{defn}
\label{def:quasi Levi} Let $M$ be a finite dimensional Lie algebra
and let $Q$ be a quasi semisimple subalgebra of $M$. We say that
$Q$ is a \emph{quasi Levi subalgebra} of $M$ if there is a solvable
ideal $P$ of $M$ such that $M=Q\oplus P$. In that case we say that
$M=Q\oplus P$ is a \emph{quasi Levi decomposition} of $M$. 
\end{defn}

Recall that $N=R\cap L$ is the nil-radical of $L=A^{(k)}$. 
\begin{prop}
\label{prop:A is strongly L=00003D=00005BA,A=00005D} Let $S$ be
a Levi subalgebra of $A$ and let $L=[A,A]$ and $Q=[S,S]$. Suppose
$p\neq2$. Then $N=[S,R]+[R,R]$, $Q$ is a quasi Levi subalgebra
of $L$ and $L=Q\oplus N$ is a quasi Levi decomposition of $L$.
Moreover, $N=[S,R]$ if $R^{2}=0$. 
\end{prop}

\begin{proof}
We have $L=[A,A]=[S\oplus R,S\oplus R]=[S,S]+[S,R]+[R,R]=Q\oplus N$
where $Q=[S,S]$ is quasi semisimple by Proposition \ref{prop:=00005BA,A=00005D quasi Levi}
and $[S,R]+[R,R]=L\cap R=N$ is the nil-radical of $L$, as required. 
\end{proof}
A subspace $B$ of $A$ is said to be a \emph{Lie inner ideal} of
$A$ if $B$ is an inner ideal of $L=A^{(-)}$, that is $[B,[B,L]]\subseteq B$.
A subspace $B$ of $A$ is said to be a \emph{Jordan inner ideal}
of $A$ if $B$ is an inner ideal of the Jordan algebra $A^{(+)}$
\cite{Lopez}. If $B^{2}=0$, then $B$ is an inner ideal of the Jordan
algebra $A^{(+)}$ if and only if it is an inner ideal of the Lie
algebra $A^{(-)}$. Indeed, since $B^{2}=0$, one has 
\begin{equation}
[b,[b',x]]=-bxb'-b'xb\label{eq:bbx}
\end{equation}
 for all $b,b'\in B$ and all $x\in A$. This justifies the following
definition. 
\begin{defn}
\label{def:Jordan-Lie } \cite{Lopez} An inner ideal $B$ of $L=A^{(k)}$
is said to be \emph{Jordan-Lie if $B^{2}=0$.} 
\end{defn}

Note that every Jordan-Lie inner ideal $B$ is abelian, i.e. $[B,B]=0$.

It follows from Benkart's result \cite[Theorem 5.1]{Benkart-1} that
if $A$ is a simple Artinian ring of characteristic not $2$ or $3$,
then every proper inner ideal of $[A,A]/(Z(A)\cap[A,A])$ is Jordan-Lie.
For $b,b'\in B$ and $x\in L$, we denote by $\{b,x,b'\}$ the Jordan
triple product
\[
\{b,x,b'\}:=bxb'+b'xb.
\]
The following lemma follows immediately from (\ref{eq:bbx}) and the
definition. 
\begin{lem}
\label{lem:Jordan-Lie } Let $L=A^{(k)}$ for some $k\ge0$ and let
$B$ be a subspace of $L$. Then $B$ is a Jordan-Lie inner ideal
of $L$ if and only if $B^{2}=0$ and $\left\{ b,x,b'\right\} \in B$
for all $b,b'\in B$ and $x\in L$. 
\end{lem}

Recall that our algebra $A$ is non-unital in general. Let $\hat{A}=A+\bbF\boldsymbol{1}_{\hat{A}}$
be the algebra obtained from $A$ by adding the external identity
element. The following lemma shows that the Jordan-Lie inner ideals
of $\hat{A}^{(k)}$ are exactly those of $A^{(k)}$ for all $k\ge0$.
\begin{lem}
\label{lem:B is J-L of A unital} Let $B$ be a subspace of $A$.
Then $B$ is a Jordan-Lie inner ideal of $A^{(k)}$ if and only if
$B$ is a Jordan-Lie inner ideal of $\hat{A}^{(k)}$ ($k\ge0$). 
\end{lem}

\begin{proof}
Note that $\hat{A}^{(k)}=A^{(k)}$ for all $k\ge1$, so we only need
to consider the case when $k=0$, i.e. $A^{(k)}=A^{(-)}$. If $B$
is a Jordan-Lie inner ideal of $A$ then $[B,[B,\hat{A}]]=[B,[B,A+\bbF\boldsymbol{1}_{\hat{A}}]]=[B,[B,A]]\subseteq B$,
so $B$ is a Jordan-Lie inner ideal of $\hat{A}$. Suppose now that
$B$ is a Jordan-Lie inner ideal of $\hat{A}$. Then $\tilde{B}=(B+A)/A$
is a Jordan-Lie inner ideal of $\hat{A}/A\cong\bbF$. Since $\tilde{B}^{2}=0$,
we get that $\tilde{B}=0$, so $B\subseteq A$. Therefore, $B$ is
a Jordan-Lie inner ideal of $A$. 
\end{proof}
We note the following standard properties of inner ideals. 
\begin{lem}
\label{lem:Bav=000026Row 2.16} Let $L$ be a Lie algebra  and let
$B$ be an inner ideal of $L$. 

(i) If $M$ is a subalgebra of $L$, then $B\cap M$ is an inner ideal
of $M$.

(ii) If $P$ is an ideal of $L$, then $(B+P)/P$ is an inner ideal
of $L/P$. Moreover, for every inner ideal $Q$ of $L/P$ there is
a unique inner ideal $C$ of $L$ containing $P$ such that $Q=C/P$.
\end{lem}

Recall that idempotents $e$ and $f$ are said to be \emph{orthogonal}
if $ef=fe=0$.
\begin{lem}
\label{lem:eAf cap Z =00003D 0} Let $A$ be a ring and let $Z(A)$
be the center of $A$. Let $e$ and $f$ be idempotents in $A$ such
that $fe=0$. Then 

(i) $eAf\cap Z(A)=0$;

(ii) $B=eAf\cap A^{(k)}$ is a Jordan-Lie inner ideal of $A^{(k)}$
for all $k\ge0$;

(iii) $eAf$ is a Jordan-Lie inner ideal of $A^{(-)}$ and of $A^{(1)}$;

(iv) $eAf=eAg$ where $g=f-ef$ is an idempotent of $A$ orthogonal
to $e$.
\end{lem}

\begin{proof}
(i) Let $z\in eAf\cap Z(A)$. Then $z=eaf$ for some $a\in A$. Since
$z\in Z(A)$, we have $0=[e,z]=[e,eaf]=eaf=z$. Therefore, $eAf\cap Z(A)=0$.

(ii) We have $B^{2}\subseteq eAfeAf=0$ and $[B,[B,A^{(k)}]\subseteq$
$BA^{(k)}B\cap A^{(k)}\subseteq eAf\cap A^{(k)}=B$, as required. 

(iii) This follows from (ii) as $eAf=[e,eAf]\subseteq[A,A]$. 

(iv) We have $g^{2}=(f-ef)^{2}=f^{2}-eff=f-ef=g$, so $g$ is an idempotent
in $A$. Since $ge=(f-ef)e=0$ and $eg=e(f-ef)=ef-ef=0$, $e$ and
$g$ are orthogonal. It remains to show that $eAf=eAg$. We have $eAg=eA(f-ef)\subseteq eAf$
and $eAf=eAf(f-ef)=eAfg\subseteq eAg$, as required.
\end{proof}

\section{Idempotent pairs }

\label{sec:Idempotent-pairs-and} The aim of this section is to prove
Theorem \ref{thm:eAf sub e'Af' imply ...}, which describes the poset
of Jordan-Lie inner ideals generated by idempotents. We start by recalling
some well known relations on the sets of idempotents.
\begin{defn}
\label{def:e<e'} Let $A$ be a ring and let $e$ and $e'$ be idempotents
in $A$. Then

(1) $e$ is said to be \emph{left dominated by} $e'$, written $e\overset{\ccL}{\le}e'$,
if $e'e=e$.

(2) $e$ is said to be \emph{right dominated by} $e'$, written $e\overset{\ccR}{\le}e'$,
if $ee'=e$.

(3) $e$ is said to be \emph{dominated by} $e'$, written $e\le e'$,
if $e$ is a left and right dominated by $e'$, that is, if $e\overset{\ccL}{\le}e'$
and $e\overset{\ccR}{\le}e'$, or equivalently, $ee'=e'e=e$.

(4) Two idempotents $e$ and $e'$ are called \emph{left equivalent},
written $e\overset{\ccL}{\sim}e'$, if $e\overset{\ccL}{\le}e'$ and
$e'\overset{\ccL}{\le}e$.

(5) Two idempotents $e$ and $e'$ are called \emph{right equivalent},
written $e\overset{\ccR}{\sim}e'$, if $e\overset{\ccR}{\le}e'$ and
$e'\overset{\ccR}{\le}e$.
\end{defn}

\begin{rem}
\label{remIdemp} (1) It is easy to see that $\overset{\ccL}{\le}$
and $\overset{\ccR}{\le}$ are \emph{preorder} relations, $\le$ is
a \emph{partial order} and $\overset{\ccL}{\sim}$ and $\overset{\ccR}{\sim}$
are \emph{equivalences. }Note that if $A$ is Artinian, then the set
of all idempotents satisfies the descending chain condition with respect
to the partial order $\le$. 

(2) If $e$ and $e'$ are idempotents in $A$, then it is easy to
check that $e\le e'$ if and only if $e'=e+e_{1}$ for some idempotent
$e_{1}$ in $A$ with $e_{1}e=ee_{1}=0$. 
\end{rem}

The following lemma is well-known.
\begin{lem}
\label{lem:A ring i) ii) iii)} Let $A$ be a ring and let $e$ and
$e'$ be idempotents in $A$. Then 

(i) $e\overset{\ccL}{\le}e'$ if and only if $eA\subseteq e'A$.

(ii) $e\overset{\ccL}{\sim}e'$ if and only if $eA=e'A$. 

(iii) If $e\overset{\ccL}{\le}e'$, then there is an idempotent $e''$
in $A$ such that $e''\le e'$ and $e''\overset{\ccL}{\sim}e$.
\end{lem}

\begin{proof}
(i) Since $e\overset{\ccL}{\le}e'$, we have $eA=e'eA\subseteq e'A$.
On the other hand, if $eA\subseteq e'A$, then $e=ee\in e'A$, so
$e'e=e$, as required.

(ii) This follows from (i). 

(iii) Put $e''=e'ee'=ee'$. Then $e''^{2}=ee'ee'=eee'=ee'=e''$, so
$e''$ is an idempotent. Since $e'e''=e'(e'ee')=e'ee'=e''$ and $e''e'=(e'ee')e'=e'ee'=e''$,
we have $e''\le e'$. It remains to note that $e''e=(ee')e=e(e'e)=ee=e$
and $ee''=e(ee')=ee'=e''$, so $e\overset{\ccL}{\sim}e''$, as required. 
\end{proof}
We say that $(e,f)$ is an \emph{idempotent pair} in $A$ if both
$e$ and $f$ are idempotents in $A$. Moreover, $(e,f)$ is \emph{orthogonal}
if $ef=fe=0$. 
\begin{defn}
\label{def:idem pair} Let $A$ be a ring and let $e$, $e'$, $f$
and $f'$ be idempotents in $A$. We say that

(1) $(e,f)$ is \emph{left-right dominated by} $(e',f')$, written
$(e,f)\overset{\ccL\ccR}{\le}(e',f')$, if $e\overset{\ccL}{\le}e'$
and $f\overset{\ccR}{\le}f'$.

(2) $(e,f)$ is\emph{ dominated by} $(e',f')$, written $(e,f)\le(e',f')$,
if $e\le e'$ and $f\le f'$.

(3) $(e,f)$ and $(e',f')$ are\emph{ left-right equivalent}, written
$(e,f)\overset{\ccL\ccR}{\sim}(e',f')$, if $(e,f)\overset{\ccL\ccR}{\le}(e',f')$
and $(e',f')\overset{\ccL\ccR}{\le}(e,f)$.
\end{defn}

Using Remark \ref{remIdemp}, we get the following. 
\begin{rem}
\label{rem:LRpreorder} (1) The relation $\overset{\ccL\ccR}{\le}$
is a preorder and $\overset{\ccL\ccR}{\sim}$ is an equivalence\emph{.
}As usual, the preorder $\overset{\ccL\ccR}{\le}$ induces a partial
order on the quotient set of idempotent pairs by the equivalence $\overset{\ccL\ccR}{\sim}$. 

(2) The relation $\le$ is a partial order. If $A$ is Artinian, then
the set of all idempotent pairs satisfies the descending chain condition
with respect to $\le$. 

(3) $(e,f)\le(e',f')$ if and only if $e'=e+e_{1}$ and $f'=f+f_{1}$
for some idempotents $e_{1}$ and $f_{1}$ in $A$ with $e$ and $e_{1}$
(resp. $f$ and $f_{1}$) orthogonal.
\end{rem}

\begin{lem}
\label{lem:(e,f) LR (e',f')} Let $A$ be a ring. Let $(e,f)$ and
$(e',f')$ be idempotent pairs in $A$ with $(e,f)\overset{\ccL\ccR}{\le}(e',f')$.
Then there is an idempotent pair $(e'',f'')$ in $A$ such that $(e'',f'')\le(e',f')$
and $(e'',f'')\overset{\ccL\ccR}{\sim}(e,f)$.
\end{lem}

\begin{proof}
This follows from Lemma \ref{lem:A ring i) ii) iii)}(iii). 
\end{proof}
\begin{prop}
\label{e'e=00003De Simple} Let $A$ be a simple ring and let $e$,
$e'$, $f$ and $f'$ be non-zero idempotents in $A$. Then we have
the following. 

(i) $eAf\ne0$. 

(ii) $eAf\subseteq e'Af'$ if and only if $(e,f)\overset{\ccL\ccR}{\le}(e',f')$. 

(iii) $eAf=e'Af'$ if and only if $(e,f)\overset{\ccL\ccR}{\sim}(e',f')$. 
\end{prop}

\begin{proof}
(i) Note that $AeA$ is a two-sided ideal of $A$ containing $e$.
Since $A$ is simple, $AeA=A$. Similarly, $AfA=A$. If $eAf=0$ then
$A^{2}=AeAAfA=AeAfA=0$, which is a contradiction.

(ii) Suppose first that $eAf\subseteq e'Af'$. Then $e'eaf=eaf$ for
all $a\in A$, so $(e'e-e)af=0$ for all $a\in A$. Hence, $e'e-e$
belongs to the left annihilator $H$ of $Af$ in $A$. Note that $H$
is a two-sided ideal of $A$. Since $A$ is simple, we have $H=A$
or $0$. As $f\not\in H$ (because $f(ff)=f\ne0$), $H=0$, so $e'e-e=0$,
or $e'e=e$. Hence, $e\overset{\ccL}{\le}e'$. Similarly, we obtain
$f\overset{\ccR}{\le}f'$. Therefore, $(e,f)\overset{\ccL\ccR}{\le}(e',f')$.
Suppose now that $(e,f)\overset{\ccL\ccR}{\le}(e',f')$. Then $e'e=e$
and $ff'=f$, so $eAf=e'eAff'\subseteq e'Af'$, as required. 

(iii) This follows from (ii).
\end{proof}
\begin{defn}
\label{def:(e,f) non-degen} (1) Let $A$ be a semisimple Artinian
ring and let $\{S_{i}\mid i\in I\}$ be the set of its simple components.
Let $e$ and $f$ be non-zero idempotents in $A$ and let $e_{i}$
(resp. $f_{i}$) be the projection of $e$ (resp. $f$) to $S_{i}$
for each $i\in I$. Then the pair $(e,f)$ is said to be\emph{ strict}
if for each $i\in I$, $e_{i}$ and $f_{i}$ are either both non-zero
or both zero.

(2) Let $A$ be an Artinian ring or a finite dimensional algebra and
let $R$ be its radical. Let $e$ and $f$ be non-zero idempotents
in $A$. We say that $(e,f)$ is \emph{strict} if $(\bar{e},\bar{f})$
is strict in $\bar{A}=A/R$. 
\end{defn}

The following lemma follows directly from the definition and Proposition
\ref{e'e=00003De Simple}(i). 
\begin{lem}
\label{prop:eAf non 0 semi} Let $A$ be a semisimple Artinian ring
and let $(e,f)$ be a non-zero strict idempotent pair in $A$. Then
$eAf\ne0$. 
\end{lem}

Now, we are ready to prove Theorem \ref{thm:eAf sub e'Af' imply ...}. 
\begin{proof}[Proof of Theorem \ref{thm:eAf sub e'Af' imply ...}]
 Recall that $(e,f)$ and $(e',f')$ are idempotent pairs in $A$
with $(e,f)$ being strict. 

(i) By Definition \ref{def:(e,f) non-degen} (2), $(\bar{e},\bar{f})$
is a strict idempotent pair in $\bar{A}$, so by Proposition \ref{prop:eAf non 0 semi},
$\bar{e}\bar{A}\bar{f}\ne0$. Therefore, $eAf\ne0$, as required. 

(ii) We need to show that $eAf\subseteq e'Af'$ if and only if $(e,f)\overset{\ccL\ccR}{\le}(e',f')$.
If $(e,f)\overset{\ccL\ccR}{\le}(e',f')$, then $eAf=e'eAff'\subseteq e'Af'$,
as required. 

Suppose now that $eAf\subseteq e'Af'$. We need only to check that
$e\overset{\ccL}{\le}e'$ (the proof for $f\overset{\ccR}{\le}f'$
is similar). Assume to the contrary that $e'e\ne e$. Then $r=e'e-e\ne0$.
Fix minimal $n\ge1$ such that $r\notin R^{n}$. By taking quotient
of $A$ by $R^{n}$ we can assume that $R^{n}=0$ and $r\in M$ where
$M=R^{n-1}$ if $n>1$ and $M=A$ (with $A$ being semisimple) if
$n=1$. Since $MR\subseteq R^{n}=0$, the right $A$-module $M$ is
actually an $\bar{A}$-module. Note that $re=(e'e-e)e=e'e-e=r$, so
$r\bar{e}=r\ne0$. Let $\{S_{i}\mid i\in I\}$ be the set of the simple
components of $\bar{A}$ and let $\bar{e}_{i}$ be the projection
of $\bar{e}$ to $S_{i}$. Since $r\bar{e}\ne0$, there is $i\in I$
such that $r\bar{e}_{i}\ne0$, so $r\bar{e}_{i}S_{i}$ is a non-zero
unital right $S_{i}$-submodule of $M$. Moreover, it is isomorphic
to a direct sum of copies of the natural $S_{i}$-module. Since $\bar{e}_{i}\ne0$
and $(e,f)$ is strict, $\bar{f}_{i}\ne0$, so $r\bar{e}S_{i}\bar{f}=r\bar{e}_{i}S_{i}\bar{f}_{i}\ne0$.
In particular, there is $a\in A$ such that $r\bar{e}\bar{a}\bar{f}\ne0$.
As $r=e'e-e$, we have that $(e'e-e)\bar{e}\bar{a}\bar{f}\ne0$, or
equivalently, $e'x\ne x$ where $x=e\bar{e}\bar{a}\bar{f}=eaf$. On
the other hand, $x\in eAf\subseteq e'Af'$, so $e'x=x$, a contradiction.
Therefore, $e\overset{\ccL}{\le}e'$, as required. 

(iii) This follows from (ii).

(iv) This follows from (ii), (iii) and Remark \ref{rem:LRpreorder}(1).

(v) This follows from (iii) and Lemma \ref{lem:(e,f) LR (e',f')}.
\end{proof}

\section{Jordan-Lie inner ideals of semisimple algebras}

Recall that $A$ is a finite dimensional associative algebra over
$\bbF$ (unless otherwise stated). If $A$ is simple then $A$ can
be identified with $\End V$ for some finite dimensional vector space
$V$ over $\bbF$. By fixing a basis of $V$ we can represent the
algebra $\End V$ in the matrix form $\ccM_{n}$, where $n=\dim V$.
We say that $\ccM_{n}$ is a \emph{matrix} \emph{realization} of $A$.
Recall that every idempotent of $\ccM_{n}$ is diagonalizable (as
its minimal polynomial is a divisor of $t^{2}-t$). Since orthogonal
idempotents commute, we get the following. 
\begin{lem}
\label{lem:e=00003Ddiag and f=00003Ddiag} Let $(e,f)$ be an orthogonal
idempotent pair in $A$. Suppose $A$ is simple. Then there is a matrix
realization of $A$ such that $e$ and $f$ can be represented by
the diagonal matrices $e=diag(1,\dots,1,0,\ldots,0)\text{ and }f=diag(0,\ldots,0,1,\dots,1)$
with $\rank(e)+\rank(f)\leq n$.
\end{lem}

Benkart proved that if A is a simple Artinian ring of characteristic
not 2 or 3, then every proper inner ideal of $[A,A]/(Z(A)\cap[A,A])$
is induced by idempotents \cite[Theorem 5.1]{Benkart-1}. We will
need a slight modification of this result.
\begin{thm}
\label{thm:B=00003DeAf if A is simple Artinian} Let $A$ be a simple
Artinian ring of characteristic not $2$ or $3$. Let $B$ be Jordan-Lie
inner ideal of $[A,A]$. Then there exists orthogonal idempotent pair
$(e,f)$ in $A$ such that $B=eAf$. 
\end{thm}

\begin{proof}
Let $Z$ be the center of $A$ and let $\hat{B}$ be the image of
$B$ in $\hat{A}=[A,A]/(Z\cap[A,A])$. Then $\hat{B}$ is a proper
inner ideal of $\hat{A}$ and by \cite[Theorem 5.1]{Benkart-1}, there
are idempotents $e$ and $f$ in $A$ with $fe=0$ such that $\hat{B}$
is the image of $eAf$ in $\hat{A}$. We wish to show that $B=eAf$.
Let $b\in B$. Then $b=eaf+z$ for some $a\in A$ and $z\in Z$. As
$B^{2}=0$ (because $B$ is Jordan-Lie), 
\[
0=b^{2}=(eaf+z)(eaf+z)=e(2az)f+z^{2}.
\]
Hence, by Lemma \ref{lem:eAf cap Z =00003D 0}(i), we obtain $z^{2}=e(-2az)f\in eAf\cap Z(A)=0$.
Therefore, $z=0$ and $B\subseteq eAf$. Conversely, let $a\in A$.
Then there is $z\in Z$ such that $eaf+z\in B$. As above, we obtain
$z=0$. Therefore, $eaf\in B$, so $B=eAf$. Since $fe=0$, by Lemma
\ref{lem:eAf cap Z =00003D 0}(iv), there is an idempotent $g$ in
$A$ such that $g$ and $e$ are orthogonal and $B=eAf=eAg$. 
\end{proof}
\begin{lem}
\label{lem:X=00003Dspan=00007Be:1<i<k<l<j<n=00007D} Let $B$ be a
Jordan-Lie inner ideal of $L=[A,A]$. Suppose $A$ is simple and $p\neq2,3$.
Then there is a matrix realization $\ccM_{n}$ of $A$ and integers
$1\leq k<l\leq n$ such that $B=span\left\{ e_{st}\mid1\leq s\leq k<l\leq t\leq n\right\} $,
where $e_{st}$ are matrix units. 
\end{lem}

\begin{proof}
This follows from Theorem \ref{thm:B=00003DeAf if A is simple Artinian}
and Lemma \ref{lem:e=00003Ddiag and f=00003Ddiag}.
\end{proof}
Recall that every simple Artinian ring $A$ is Von\textbf{ }Neumann\textbf{
}regular, i.e. for every $x\in A$ there is $y\in A$ such that $x=xyx$
\cite{Goodearl}. 
\begin{lem}
\label{lem:B=00003D=00005BB,=00005BB,A=00005D=00005D if A is simple Artinian}
Let $A$ be a simple Artinian ring of characteristic not $2$ or $3$
and let $B$ be a Jordan-Lie inner ideal of $A^{(1)}$. Then $B=[B,[B,A^{(1)}]]$. 
\end{lem}

\begin{proof}
We need only to show that $B\subseteq[B,[B,A^{(1)}]]$. Let $b\in B$.
By Theorem \ref{thm:B=00003DeAf if A is simple Artinian}, $B=eAf$
for some orthogonal idempotents $e$ and $f$ in $A$, so $b=eaf$
for some $a\in A$. Since $A$ is Von Neumann regular, $b=bxb$ for
some $x\in A$. Hence, $eaf=b=bxb=(eaf)x(eaf)$. Put $y=fxe=[f,fxe]\in A^{(1)}$.
Then $b=byb$, so $[b,[b,y]]=-2byb=-2b$. This implies $b\in[B,[B,A^{(1)}]]$,
as required. 
\end{proof}
Let $L$ be a finite dimensional semisimple Lie algebra and let $\{L_{i}\mid i\in I\}$
be the set of the simple components of $L$. If $B$ is an inner ideal
of $L$ and the ground field is of characteristic $p\neq2,3,5,7$
then $B=\bigoplus_{i\in I}B_{i}$, where $B_{i}=B\cap L_{i}$ (see
\cite[Proposition 2.3]{Lop=000026Gar=000026Loz-1}). As the following
lemma shows we need less restrictions on $p$ if $L=[A,A]$ and $B$
is Jordan-Lie. 
\begin{lem}
\label{lem:B=00003DB1+...} Suppose $A$ is semisimple and $p\neq2,3$.
Let $\{S_{i}\mid i\in I\}$ be the set of the simple components of
$A$ and let $B$ be a Jordan-Lie inner ideal of $L=[A,A]$. Then
$B=\bigoplus_{i\in I}B_{i}$, where $B_{i}=B\cap S_{i}$ is a Jordan-Lie
inner ideal of $L_{i}=[S_{i},S_{i}]$. 
\end{lem}

\begin{proof}
Let $\psi_{i}:L\rightarrow L_{i}$, $\psi_{i}((x_{1},\ldots,x_{i},\ldots)=x_{i}$,
be the natural projection. We need to show that $\psi_{i}(B)=B_{i}$.
By Lemma \ref{lem:Bav=000026Row 2.16}, $\psi_{i}(B)$ is a Jordan-Lie
inner ideal of $L_{i}$. Clearly, $B_{i}\subseteq\psi_{i}(B)$. On
the other hand, by Lemma \ref{lem:B=00003D=00005BB,=00005BB,A=00005D=00005D if A is simple Artinian}
\[
\psi_{i}(B)=[\psi_{i}(B),[\psi_{i}(B),L_{i}]]\subseteq[B,[B,L_{i}]]\subseteq B\cap L_{i}\subseteq B_{i}
\]
 for all $i\in I$. Therefore, $B=\bigoplus_{i\in I}B_{i}$. Since
$B\subseteq[A,A]$ we have $B_{i}\subseteq[S_{i},S_{i}]$, as required. 
\end{proof}
The following proposition first appeared in \cite[Lemma 6.6]{Rowley}
in the case $p=0$. 
\begin{lem}
\label{B=00003DeAf if A is semi Artinian } Suppose $A$ is semisimple
and $p\neq2,3$. Let $B$ be a Jordan-Lie inner ideal of $L=[A,A]$.
Then there exists a strict orthogonal idempotent pair $(e,f)$ in
$A$ such that $B=eAf$. 
\end{lem}

\begin{proof}
Let $\{S_{i}\mid i\in I\}$ be the set of the simple components of
$A$. Using Theorem \ref{thm:B=00003DeAf if A is simple Artinian}
and Lemma \ref{lem:B=00003DB1+...} we get that $B=\bigoplus_{i\in I}e_{i}S_{i}f_{i}$
for some orthogonal idempotent pairs $(e_{i},f_{i})$ in $S_{i}$.
Moreover, we can assume that $e_{i}=f_{i}=0$ if $B_{i}=B\cap S_{i}=0$.
Put $e=\sum_{i\in I}e_{i}$ and $f=\sum_{i\in I}f_{i}$. Then $(e,f)$
is a strict orthogonal idempotent pair in $A$ and $eAf=\bigoplus_{i\in I}e_{i}S_{i}f_{i}=B$,
as required. 
\end{proof}
\begin{lem}
\label{lem:If A semi, B of A(-) =000026 =00005BA,A=00005D } Suppose
$A$ is semisimple and $p\neq2,3$. Let $B$ be a Jordan-Lie inner
ideal of $A^{(-)}$. Then $B$ is a Jordan-Lie inner ideal of $[A,A]$. 
\end{lem}

\begin{proof}
Let $b\in B$. Since $A$ is Von Neumann regular, there is $x\in A$
such that $b=bxb$. As $b^{2}=0$, 
\[
b=bxb=b(xb)-(xb)b=[b,xb]\in[A,A].
\]
Therefore, $B\subseteq[A,A]$, so $B$ is a Jordan-Lie inner ideal
of $[A,A]$. 
\end{proof}
Lemmas \ref{B=00003DeAf if A is semi Artinian } and \ref{lem:If A semi, B of A(-) =000026 =00005BA,A=00005D }
imply that all Jordan-Lie inner ideals of $A^{(-)}$ are generated
by idempotents, which is essentially known, see for example \cite[Theorem 6.1(2)]{Lopez}.
We summarize description of Jordan-Lie inner ideals of $A^{(k)}$
in the following proposition.
\begin{prop}
\label{cor:JL=00003DeAf} Suppose $A$ is semisimple, $p\neq2,3$
and $k\ge0$. Let $B$ be a subspace of $A$. Then $B$ is a Jordan-Lie
inner ideal of $A^{(k)}$ if and only if $B=eAf$ where $(e,f)$ is
a strict orthogonal idempotent pair in $A$. 
\end{prop}

\begin{proof}
By Proposition \ref{prop:=00005BA,A=00005D quasi Levi}, $A^{(k)}=A^{(1)}$
if $k\ge1$. The ``only if'' part now follows from Lemmas \ref{B=00003DeAf if A is semi Artinian }
($k\ge1)$ and \ref{lem:If A semi, B of A(-) =000026 =00005BA,A=00005D }
($k=0$), and the ``if'' part follows from Lemma \ref{lem:eAf cap Z =00003D 0}(iii).
\end{proof}
\begin{rem}
Proposition \ref{cor:JL=00003DeAf} (as well as Lemma \ref{lem:B is L-perfect if A is semi}
below) can also be easily deduced from a general result by Fernández
López \cite[Theorem 6.1(2)]{Lopez}, which essentially classifies
all inner ideals of semiprime rings coinciding with their socles.
\end{rem}

\section{$L$-Perfect inner ideals }

\subsection*{$\boldsymbol{1}$-perfect associative algebras and their associated
Lie algebras}
\begin{defn}
\label{def:Lie solvable } The associative algebra $A$ is said to
be \emph{Lie solvable} if the Lie algebra $A^{(-)}$ is solvable. 
\end{defn}

The following is well known. 
\begin{lem}
\label{series} Let $p\ne2$. Then the following are equivalent. 

(i) $A$ is Lie solvable. 

(ii) There is a descending chain of ideals $A=A_{0}\supset A_{1}\supset\dots\supset A_{r}=\{0\}$
of $A$ such that $\dim A_{i}/A_{i+1}=1$ for $0\leq i\leq r-1$. 

(iii) There is a descending chain of subalgebras $A=A_{0}\supset A_{1}\supset\dots\supset A_{r}=\{0\}$
of $A$ such that $A_{i+1}$ is an ideal of $A_{i}$ and $\dim A_{i}/A_{i+1}=1$
for $0\leq i\leq r-1$. 
\end{lem}

\begin{proof}
The implications $(ii)\Rightarrow(iii)$ and $(iii)\Rightarrow(i)$
are obvious (as $A^{(i)}\subseteq A_{i}$ for all $i$). To prove
$(i)\Rightarrow(ii)$, suppose that $A$ is Lie solvable. Let $R$
be the radical of $A$ and let $S=A/R$. Then $S$ is a Lie solvable
semisimple algebra, so by Lemma \ref{lem:sl_n} and Proposition \ref{prop:=00005BA,A=00005D quasi Levi},
$S\cong\bbF^{m}$ direct sum of $m$ copies of $\bbF$ for some $m$.
If $S=0$, then $A=R$ is nilpotent, so such a chain exists. Suppose
that $S\neq0$. Since all simple components of $S$ are 1-dimensional,
all composition factors of the $S$-bimodule $R/R^{2}$ are one-dimensional,
so there is a chain of ideals in $A/R^{2}$ with 1-dimensional quotients.
The lemma now follows by induction on the degree of nilpotency of
$R$. 
\end{proof}
\begin{defn}
An associative algebra is said to be $1$-\emph{perfect} if it has
no ideals of codimension 1. An ideal is said to be $1$-\emph{perfect}
if it is $1$-perfect as an algebra.
\end{defn}

By using 2nd and 3rd Isomorphism Theorems we get the following properties
of $1$-perfect ideals.
\begin{lem}
\label{1p} (i) The sum of $1$-perfect ideals is $1$-perfect.

(ii) If $P$ is a $1$-perfect ideal of $A$ and $Q$ is a $1$-perfect
ideal of $A/P$ then the full preimage of $Q$ in $A$ is a $1$-perfect
ideal of $A$.
\end{lem}

Lemma \ref{1p}(i) implies that every algebra has the largest $1$-perfect
ideal. 
\begin{defn}
\label{def:1-perfect radical} The largest $1$-perfect ideal $\ccP_{1}(A)$
of $A$ is called the\emph{ }$1$\emph{-perfect radical }of $A$.
\end{defn}

The following proposition shows that $\ccP_{1}(A)$ has radical-like
properties indeed. 
\begin{prop}
\label{1pproperties} (i) $\ccP_{1}(A)^{2}=\ccP_{1}(A)$.

(ii) $\ccP_{1}(\ccP_{1}(A))=\ccP_{1}(A)$.

(iii) $\ccP_{1}(A/\ccP_{1}(A))=0$.

(iv) Let $A=A_{0}\supset A_{1}\supset\dots\supset A_{r}$ be any maximal
chain of subalgebras of $A$ such that $A_{i+1}$ is an ideal of $A_{i}$
and $\dim A_{i}/A_{i+1}=1$ for $0\leq i\leq r-1$. Then $A_{r}=\ccP_{1}(A)$.

(v) If $p\ne2$ then $\ccP_{1}(A)=0$ if and only if $A$ is Lie solvable.

(vi) If $p\ne2$ then $\ccP_{1}(A)$ is the smallest ideal of $A$
such that $A/\ccP_{1}(A)$ is Lie solvable.
\end{prop}

\begin{proof}
(i) Let $P=\ccP_{1}(A)$. We need to show $P^{2}=P$. Let $Q=P/P^{2}$.
Then $Q$ is a finite dimensional algebra with $Q^{2}=0$. Suppose
$Q\ne0$. Then $Q$ has an ideal $N$ of codimension 1 and the full
preimage of $N$ in $P$ is an ideal of $P$ of codimension 1. Thus,
$P$ is not 1-perfect, which is a contradiction.

(ii) is obvious.

(iii) follows from Lemma \ref{1p}(ii).

(iv) Note that $A_{r}$ is 1-perfect. This implies that $A_{r}^{2}=A_{r}$
and $A_{r}$ is actually a 1-perfect ideal of $A$: 
\[
AA_{r}=AA_{r}A_{r}\dots A_{r}\subseteq A_{0}A_{1}A_{2}\dots A_{r}\subseteq A_{r}
\]
 (and similarly $A_{r}A\subseteq A_{r}$). Thus $A_{r}\subseteq\ccP_{1}(A)$.
Fix maximal $k$ such that $\ccP_{1}(A)\subseteq A_{k}$. If $k<r$
then $A_{k+1}\cap\ccP_{1}(A)$ is an ideal of $\ccP_{1}(A)$ of codimension
$1$, a contradiction, so $k=r$ and $A_{r}=\ccP_{1}(A)$, as required.

(v) and (vi) follow from (iv) and Lemma \ref{series}.
\end{proof}
Importance of $1$-perfect algebras is shown by the following result
from \cite{Baranov1}. 
\begin{thm}
\label{BaranovLiePerfect} \cite[Theorem 1.1(1)]{Baranov1} If $A$
is $1$-perfect and $p\ne2$, then $[A,A]$ is a perfect Lie algebra.
\end{thm}

Combining this result with Proposition \ref{1pproperties}(vi) we
get the following.
\begin{lem}
\label{lem:A^(k)=00003Dp(A)} Let $p\ne2$. Then $A^{(\infty)}=\ccP_{1}(A)^{(1)}$.
\end{lem}

\begin{proof}
Since $A/\ccP_{1}(A)$ is Lie solvable, there is $n\ge0$ such that
$(A/\ccP_{1}(A))^{(n)}=0$, so $A^{(n+1)}\subseteq\ccP_{1}(A)^{(1)}$.
As $\ccP_{1}(A)$ is $1$-perfect, by Theorem \ref{BaranovLiePerfect},
$\ccP_{1}(A)^{(1)}$ is perfect, so $A^{(\infty)}=A^{(n+1)}=\ccP_{1}(A)^{(1)}$. 
\end{proof}

\subsection*{$L$-perfect inner ideals}
\begin{defn}
\label{def:Perfect} Let $L$ be a Lie algebra and let $B$ be an
inner ideal of $L$. We say that $B$ is \emph{$L$-perfect} if $B=[B,[B,L]]$. 
\end{defn}

It is known that every inner ideal of a semisimple Lie algebra $L$
is $L$-perfect if $p\neq2,3,5,7$, see for example \cite[Proposition 2.3]{Lop=000026Gar=000026Loz-1}
(or \cite[Lemmas 2.19 and 2.20]{Bav=000026Row} for characteristic
zero). As the following lemma shows we need less restrictions on $p$
if $L=[A,A]$ and $B$ is Jordan-Lie. 
\begin{lem}
\label{lem:B is L-perfect if A is semi} Suppose $A$ is semisimple,
$k\ge0$ and $p\neq2,3$. Then every Jordan-Lie inner ideal of $L=A^{(k)}$
is $L$-perfect. 
\end{lem}

\begin{proof}
Suppose first that $k\ge1$. Then $A^{(k)}=A^{(1)}$ by Proposition
\ref{prop:=00005BA,A=00005D quasi Levi}. Therefore, this follows
from Lemma \ref{lem:B=00003DB1+...} and Lemma \ref{lem:B=00003D=00005BB,=00005BB,A=00005D=00005D if A is simple Artinian}. 

Suppose now that $k=0$. Let $B$ be a Jordan-Lie inner ideal of $A^{(-)}$.
Then by Lemma \ref{lem:If A semi, B of A(-) =000026 =00005BA,A=00005D },
$B$ be is Jordan-Lie inner ideal of $A^{(1)}$, so $B$ is $A^{(1)}$-perfect
by above. This obviously implies that $B$ is $A^{(-)}$-perfect. 
\end{proof}
\begin{lem}
\label{lem:B sub L^(infty)} Let $L$ be a Lie algebra and let $B$
be an inner ideal of $L$. If $B$ is $L$-perfect, then $B$ is an
inner ideal of $L^{(k)}$ for all $k\ge0$.
\end{lem}

\begin{proof}
Suppose $B\subseteq L^{(k)}$ for some $k\ge0$. Then 
\[
B=[B,[B,L]]\subseteq[L^{(k)},[L^{(k)},L]]\subseteq[L^{(k)},L^{(k)}]=L^{(k+1)},
\]
so the result follows by induction on $k$.
\end{proof}
\begin{lem}
\label{lem:B sub P(A)} Let $B$ be an $L$-perfect Jordan-Lie inner
ideal of $L=A^{(k)}$ ($k\ge0$). If $p\ne2$ then $B\subseteq\ccP_{1}(A)$
and $B$ is a Jordan-Lie inner ideal of $\ccP_{1}(A)^{(1)}$. 
\end{lem}

\begin{proof}
Since $B$ is $L$-perfect, by Lemma \ref{lem:B sub L^(infty)}, $B\subseteq L^{(\infty)}=A^{(\infty)}$,
so $B$ is a Jordan-Lie inner ideal of $A^{(\infty)}$. It remains
to note that $A^{(\infty)}=\ccP_{1}(A)^{(1)}$ by Lemma \ref{lem:A^(k)=00003Dp(A)}. 
\end{proof}

\subsection*{The core of inner ideals}

Let $B$ be an inner ideal of $L$. Then $[B,[B,L]]\subseteq B$.
It is well known that $[B,[B,L]]$ is an inner ideal of $L$ (see
for example \cite[Lemma 1.1]{Benkart}). Put $B_{0}=B$ and consider
the following inner ideals of $L$: 
\begin{equation}
B_{n}=[B_{n-1},[B_{n-1},L]]\subseteq B_{n-1}\quad\text{for}\quad n\geq1.\label{eq:Bn}
\end{equation}
Then $B=B_{0}\supseteq B_{1}\supseteq B_{2}\supseteq\ldots$. As $L$
is finite dimensional, this series terminates. This motivates the
following definition. 
\begin{defn}
\label{def:core} Let $L$ be a finite dimensional Lie algebra and
let $B$ be an inner ideal of $L$. Then there is an integer $n$
such that $B_{n}=B_{n+1}$. We say that $B_{n}$ is \emph{the} \emph{core}
\emph{of} $B$, denoted by $\core_{L}(B)$. 
\end{defn}

\begin{lem}
\label{lem:cor is perfect} Let $L$ be a finite dimensional Lie algebra
and let $B$ be an inner ideal of $L$. Then

(i) $\core_{L}(B)$ is $L$-perfect;

(ii) $B$ is $L$-perfect if and only if $B=\core_{L}(B)$;

(iii) $\core_{L}(B)$ is an inner ideal of $L^{(k)}$ for all $k\ge0$.
\end{lem}

\begin{proof}
(i) and (ii) follow from Definitions \ref{def:Perfect} and \ref{def:core}. 

(iii) follows from (i) and Lemma \ref{lem:B sub L^(infty)}.
\end{proof}
\begin{rem}
Let $k\ge0$. If $S$ is a Levi subalgebra of $A$, then $A=S\oplus R$,
so $A^{(k)}=S^{(k)}\oplus N$, where $N=R\cap A^{(k)}$. Moreover,
$\bar{A}^{(k)}=A^{(k)}/N=A^{(k)}/R\cap A^{(k)}$ is the image of $A^{(k)}$
in $\text{\ensuremath{\bar{A}=A/R}}$. 
\end{rem}

\begin{lem}
\label{B=00003Dcor(B)} Let $B$ be a Jordan-Lie inner ideal of $L=A^{(k)}$
($k\ge0$). If $p\neq2,3$, then 

(i) $\bar{B}=\overline{\core_{L}(B)}$.

(ii) If $\core_{L}(B)=0$, then $B\subset N$.
\end{lem}

\begin{proof}
(i) Since $\bar{A}$ is semisimple and $\bar{B}$ is a Jordan-Lie
inner ideal of $\bar{L}=\bar{A}^{(k)}$, by Lemma \ref{lem:B is L-perfect if A is semi},
$\bar{B}$ is $\bar{L}$-perfect. Hence, by Lemma \ref{lem:cor is perfect},
$\bar{B}=\core_{\bar{L}}(\bar{B})=\overline{\core_{L}(B)}$.

(ii) This follows from (i).
\end{proof}

\section{Bar-minimal and regular inner ideals}

Recall that $L=A^{(k)}$ for some $k\ge0$, $N=R\cap L$, and $\bar{B}$
is the image of a subspace $B$ of $L$ in $\bar{L}=L+R/R\cong L/N$. 

\subsection*{Bar-minimal inner ideals}
\begin{defn}
\label{def:minimal} Let $L=A^{(k)}$ and let $X$ be an inner ideal
of $\bar{L}$. Suppose that $B$ is an inner ideal of $L$. We say
that $B$ is \emph{$X$-minimal} (or simply, \emph{bar-minimal})\emph{
}if $\bar{B}=X$ and for every inner ideal $B'$ of $L$ with $\bar{B}'=X$
and $B'\subseteq B$ one has $B'=B$. 
\end{defn}

\begin{lem}
\label{lem: B =00003D =00005BB,=00005BB,L=00005D=00005D } Let $k\ge0$
and let $B$ be a Jordan-Lie inner ideal of $L=A^{(k)}$. Suppose
that $B$ is bar-minimal and $p\neq2,3$. Then the following hold. 

(i) $B=\core_{L}B$. 

(ii) $B$ is $L$-perfect.

(iii) $B$ is a Jordan-Lie inner ideal of $L^{(m)}=A^{(k+m)}$ for
all $m\ge0$.
\end{lem}

\begin{proof}
(i) By definition of the core, $\core_{L}(B)$ is an inner ideal of
$L$ contained in $B$. By Lemma \ref{B=00003Dcor(B)}(i), $\overline{\core_{L}(B)}=\bar{B}$.
Since $B$ is bar-minimal, we have $B=\core_{L}B$. 

(ii) This follows directly from (i) and Lemma \ref{lem:cor is perfect}(i). 

(iii) This follows from (ii) and Lemma \ref{lem:B sub L^(infty)}. 
\end{proof}
Recall that a Lie algebra $L$ is said to be perfect if $L=[L,L]$.
We will need the following result.
\begin{lem}
\label{lem:B=00003DB1+B2 if L=00003DL1+L2} Let $L$ be a perfect
Lie algebra and let $B$ be an $L$-perfect inner ideal of $L$. Suppose
that $L=\bigoplus_{i\in I}L_{i}$, where each $L_{i}$ is an ideal
of $L$. Then $B=\bigoplus_{i\in I}B_{i}$, where $B_{i}=B\cap L_{i}$.
Moreover, if $L=A^{(k)}$ ($k\ge0$) and $B$ is bar-minimal then
$B_{i}$ is a $\bar{B}_{i}$-minimal inner ideal of $L_{i}$, for
all $i\in I$.
\end{lem}

\begin{proof}
Note that $[B,[B,L_{i}]]\subseteq B\cap L_{i}=B_{i}$, for all $i\in I$.
Therefore, 
\[
B=[B,[B,L]]=\sum_{i\in I}[B,[B,L_{i}]]\subseteq\sum_{i\in I}B_{i}\subseteq B,
\]
so $B=\sum_{i\in I}B_{i}$. As $B_{i}\cap B_{j}\subseteq L_{i}\cap L_{j}=0$
for all $i\neq j$, $B=\bigoplus_{i\in I}B_{i}$. Clearly, if $B$
is bar-minimal, then each $B_{i}$ is $\bar{B}_{i}$-minimal. 
\end{proof}

\subsection*{Split inner ideals }

Let $L$ be a Lie algebra and let $Q$ be a subalgebra of $L$. Recall
that $Q$ is said to be a quasi Levi subalgebra of $L$ if $Q$ is
quasi semisimple and there is a solvable ideal $P$ of $L$ such that
$L=Q\oplus P$.
\begin{defn}
\label{def:splitness-1} Let $L$ be a finite dimensional Lie algebra
and let $B$ be a subspace of $L$. Suppose that there is a quasi
Levi decomposition $L=Q\oplus N$ of $L$ such that $B=B_{Q}\oplus B_{N}$,
where $B_{Q}=B\cap Q$ and $B_{N}=B\cap N$. Then we say that $B$
\emph{splits} \emph{in} $L$ and $Q$ is a $B$\emph{-splitting} \emph{quasi
Levi subalgebra }of $L$. 
\end{defn}

\begin{defn}
\label{def:splitness} Let $B$ be a subspace of $A$. Suppose that
there is a Levi subalgebra $S$ of $A$ such that $B=B_{S}\oplus B_{R}$,
where $B_{S}=B\cap S$ and $B_{R}=B\cap R$. Then we say that $B$
\emph{splits} in $A$ and $S$ is a $B$\emph{-splitting} Levi subalgebra\emph{
}of $A$.
\end{defn}

\begin{lem}
\label{lem:If split in A, then split in L} Let $L=A^{(k)}$ ($k\ge1$)
and let $B$ be a subspace of $L$. Suppose $p\ne2$. If $B$ splits
in $A$, then $B$ splits in $L$. 
\end{lem}

\begin{proof}
Suppose that $B$ splits in $A$. Then there is a $B$-splitting Levi
subalgebra $S$ of $A$ such that $B=B_{S}\oplus B_{R}$, where $B_{S}=B\cap S$
and $B_{R}=B\cap R$. Clearly, $Q=[S,S]=S^{(k)}$ is a quasi semisimple
subalgebra of $L$, $N=L\cap R$ is a solvable ideal of $L$, and
$L=Q\oplus N$ is a quasi Levi decomposition of $L$. It is easy to
see that $B_{S}\subseteq Q$ and $B_{R}\subseteq N$, so $B$ splits
in $L$. 
\end{proof}
\begin{lem}
\label{lem:eAf splits} Let $B$ be an inner ideal of $L=A^{(k)}$
($k\ge0$). Suppose $B=eAf$ for some orthogonal idempotents $e$
and $f$ of $A$. Then (i) $B$ splits in $A$ and (ii) if $k\ge1$
then $B$ splits in $L$. 
\end{lem}

\begin{proof}
(i) Since $e$ and $f$ are orthogonal, By Wedderburn-Malcev theorem
there is a Levi subalgebra $S$ of $A$ such that $e,f\in S$. Thus,
$B=eAf=e(S\oplus R)f=eSf\oplus eRf$ as required.

(ii) This follows directly from (i) and Lemma \ref{lem:If split in A, then split in L}.
\end{proof}
\begin{prop}
\label{prop:C sub B splits in A} Let $C\subseteq B$ be subspaces
of $A$ such that $\bar{C}=\bar{B}$. If $C$ splits in $A$, then
$B$ splits in $A$.
\end{prop}

\begin{proof}
Suppose $C$ splits in $A$. Then there exists a Levi subalgebra $S$
of $A$ such that $C=C_{S}\oplus C_{R}$, where $C_{S}=C\cap S$ and
$C_{R}=C\cap R$. Put $B_{S}=B\cap S$ and $B_{R}=B\cap R$. Since
$\bar{C}=\bar{B}$, we have $B\subseteq C+R=C_{S}+R$ and $C_{S}=C\cap S\subseteq B_{S}\subseteq B$.
Now by Modular Law, 
\[
B=B\cap(C_{S}+R)\subseteq C_{S}+B\cap R\subseteq B_{S}+B_{R}\subseteq B.
\]
Thus, $B=B_{S}\oplus B_{R}$ as required.
\end{proof}
\begin{cor}
\label{cor:cor(B) splits} Let $L=A^{(k)}$ ($k\ge0$) and let $B$
be an inner ideal of $L$. Suppose that $p\ne2,3$. If $\core_{L}(B)$
splits in $A$, then $B$ splits in $A$.
\end{cor}

\begin{proof}
By Lemma \ref{lem:cor is perfect}, $\overline{\core_{L}(B)}=\bar{B}$.
Since $\core_{L}(B)\subseteq B$ and $\core_{L}(B)$ splits, by Proposition
\ref{prop:C sub B splits in A}, $B$ splits. 
\end{proof}
\begin{defn}
\label{def:large } Let $G$ be a subalgebra of $A$. We say that
$G$ is \emph{large in $A$ }if $G+R=A$ (equivalently, there is a
Levi subalgebra $S$ of $A$ such that $S\subseteq G$; or equivalently,
$G/\rad G$ is isomorphic to $A/R$).
\end{defn}

\begin{rem}
\label{rem:rad_large} Let $G$ be a large subalgebra of $A$ and
let $B$ be a subspace of $\ccP_{1}(G)$. Then $\rad(G)=G\cap R$
and $\rad(\ccP_{1}(G))=\ccP_{1}(G)\cap\rad(G)=\ccP_{1}(G)\cap R$,
so the image $\bar{B}$ of $B$ in $A/R$ is isomorphic to the images
of $B$ in $G/\rad(G)$ and $\ccP_{1}(G)/\rad(\ccP_{1}(G))$, respectively.
Thus, we can use the same notation $\bar{B}$ for the images of $B$
in all these quotient spaces.
\end{rem}

\begin{prop}
\label{lem:Large is splits} Let $B$ be a subspace of $A$. Let $G$
be a large subalgebra of $A$ and let $C$ be a subspace of $\ccP_{1}(G)$.
Suppose that $C\subseteq B$, $\bar{C}=\bar{B}$, and $C$ splits
in $\ccP_{1}(G)$. Then $B$ splits in $A$.
\end{prop}

\begin{proof}
Put $R_{1}=\rad\ccP_{1}(G)$. By Remark \ref{rem:rad_large}, $R_{1}\subseteq\rad(G)\subseteq R$.
Let $S_{1}$ be a $C$-splitting Levi subalgebra of $\ccP_{1}(G)$,
so $C=C_{S_{1}}\oplus C_{R_{1}}$, where $C_{S_{1}}=C\cap S_{1}$
and $C_{R_{1}}=C\cap R_{1}$. Note that $S_{1}$ is a semisimple subalgebra
of $A$, so by Wedderburn-Malcev Theorem there is a Levi subalgebra
$S$ of $A$ such that $S_{1}\subseteq S$. Since $S_{1}\subseteq S$
and $R_{1}\subseteq R$, $C$ splits in $A$, so the result follows
from Proposition \ref{prop:C sub B splits in A}. 
\end{proof}
Since $A$ is large in $A$, we get the following corollary. 
\begin{cor}
\label{cor:CB splits} Let $B$ be a subspace of $A$ and let $C$
be a subspace of $\ccP_{1}(A)$. Suppose that $C\subseteq B$, $\bar{C}=\bar{B}$,
and $C$ splits in $\ccP_{1}(A)$. Then $B$ splits in $A$.
\end{cor}

\begin{prop}
\label{prop:Cbar} Let $B$ be a Jordan-Lie inner ideal of $L=A^{(k)}$
($k\ge0$). Let $G$ be a large subalgebra of $A$ and let $B'=B\cap G^{(k)}$.
Suppose $p\ne2,3$ and $\bar{B}'=\bar{B}$. Put $C=\core_{G^{(k)}}(B')$.
Then $C$ is a Jordan-Lie inner ideal of $\ccP_{1}(G)^{(1)}$ such
that $C\subseteq B$ and $\bar{C}=\bar{B}$. 
\end{prop}

\begin{proof}
Note that $B'=B\cap G^{(k)}$ is a Jordan-Lie inner ideal of $G^{(k)}$.
By Lemmas \ref{lem:cor is perfect}(i) and \ref{B=00003Dcor(B)}(i),
$C=\core_{G^{(k)}}(B')$ is a $G^{(k)}$-perfect Jordan-Lie inner
ideal of $G^{(k)}$ with $C\subseteq B'\subseteq B$ and $\bar{C}=\bar{B}'=\bar{B}$.
It remains to note that by Lemma \ref{lem:B sub P(A)}, $C$ is Jordan-Lie
inner ideal of $\ccP_{1}(G)^{(1)}$. 
\end{proof}

\subsection*{Regular inner ideals }

In this subsection we describe bar-minimal regular inner ideals of
$A^{(k)}$ ($k\ge0$). We start with the following result which is
a slight generalisation of \cite[Lemma 4.1]{Bav=000026Row}. 
\begin{lem}
\label{lem:Bar=000026Row 4.1} Let $L=A^{(k)}$ for some $k\ge0$
and let $B$ be a subspace of $L$ such that $B^{2}=0$. Then the
following hold. 

(i) If $p\ne2$ then $B$ is an inner ideal of $L$ if and only if
$bLb\subseteq B$ for all $b\in B$. 

(ii) $BAB\subseteq L\cap A^{(1)}$.

(iii) If $BAB\subseteq B$, then $B$ is a Jordan-Lie inner ideal
of $L$.
\end{lem}

\begin{proof}
(i) This follows from Lemma \ref{lem:Jordan-Lie } as 
\[
\left\{ b,x,b'\right\} =bxb'+b'xb=(b+b')x(b+b')-bxb-b'xb'.
\]

(ii) $bxb'=[b,xb']\in[A^{(k)},A]\subseteq A^{(k)}\cap A^{(1)}=L\cap A^{(1)}$
for all $b,b'\in B$ and $x\in L$. 

(iii) This is obvious as $[B,[B,L]]\subseteq BAB$. 
\end{proof}
\begin{defn}
\label{def:Regular} Let $B$ be a subspace of $L=A^{(k)}$ ($k\ge0$).
Then $B$ is said to be a \emph{regular }inner ideal of \emph{$L$}
(with respect to $A$) if $B^{2}=0$ and $BAB\subseteq B$.
\end{defn}

Regular inner ideals were first defined in \cite{Bav=000026Row} (in
characteristic zero) and were recently used in \cite{Bav=000026Lop}
to classify maximal zero product subsets of simple rings. Note that
every regular inner ideal is Jordan-Lie (see Lemma \ref{lem:Bar=000026Row 4.1}).
However, the converse is not true as the following examples show.
\begin{example}
\label{nr-1} Let $A=S\oplus R$ be a finite dimensional associative
algebra with a Levi subalgebra $S$ and radical $R$ such that $S\cong\ccM_{n}$
with $n\ge3$, $R^{2}=0$ and $R\cong S$ as an $S$-bimodule. Let
$\{e_{ij}\mid1\le i\le j\le n\}$ and $\{f_{ij}\mid1\le i\le j\le n\}$
be the standard bases of $S$ and $R$, respectively, consisting of
matrix units. Fix any $1<s,t<n$. Let $B=span\{e_{1n},f_{1n},b\}$,
where $b=f_{1s}+f_{tn}$. Then $B^{2}=0$ and it is easy to check
that $[B,[B,A]]=B$, so $B$ is a Jordan-Lie inner ideal of $A$.
However, $BAB\nsubseteq B$ because $e_{1n}e_{n1}b=f_{1s}\notin B$,
so $B$ is not regular. 
\end{example}

\begin{example}
\label{nr} Let $\mathfrak{n}_{4}\subset\ccM_{4}$ be the set of all
strictly upper triangular $4\times4$ matrices over $\bbF$. Let $A$
be the direct sum of two nilpotent ideals $T_{4}$ and $T'_{4}$ with
both of them isomorphic to $\mathfrak{n}_{4}$. Clearly, $A^{4}=0$.
Let $\{e_{ij}\mid1\le i<j\le4\}$ and $\{e'_{ij}\mid1\le i<j\le4\}$
be the standard bases of $T_{4}$ and $T'_{4}$, respectively, consisting
of matrix units. Consider the following elements of $A$:
\[
b_{1}=e_{12}+e'_{34},\quad b_{2}=e_{34}+e'_{12},\quad a=e_{23}+e'_{23},\quad b=e_{14}+e'_{14}.
\]

Let $A_{1}=A^{2}+span\{b_{1},b_{2},a\}$. Then $A_{1}$ is a subalgebra
of $A$ as $A_{1}^{2}\subseteq A^{2}\subset A_{1}$. Consider the
subspace $B=span\{b_{1},b_{2},b\}$ of $A_{1}$. It is easy to check
that $B^{2}=0$ and $B$ is a Jordan-Lie inner ideal of $A_{1}$.
Moreover, $B$ is not regular as $b_{1}ab_{2}=e_{14}\notin B$. 

Note that $B$ is also a non-regular Jordan-Lie inner ideal of the
unital algebra $\hat{A}_{1}=A_{1}+\bbF\boldsymbol{1}_{\hat{A}}$,
by Lemma \ref{lem:B is J-L of A unital}.
\end{example}

\begin{lem}
\label{eAf is regular} Let $A$ be any ring and let $e$ and $f$
be idempotents in $A$ with $fe=0$. Then the following hold. 

(i) If $eAf\subseteq A^{(k)}$ ($k\ge0$), then $eAf$ is a regular
inner ideal of $A^{(k)}$. 

(ii) $eAf$ is a regular inner ideal of $A^{(-)}$ and $A^{(1)}$.
\end{lem}

\begin{proof}
(i) By Lemma \ref{lem:eAf cap Z =00003D 0}(ii), $eAf$ is a Jordan-Lie
inner ideal of $A^{(k)}$ ($k\ge0$). It remains to note that $(eAf)A(eAf)\subseteq eAf$. 

(ii) This follows from (i) and Lemma \ref{lem:eAf cap Z =00003D 0}(iii). 
\end{proof}
The following result is proved in \cite[Proposition 4.12]{Bav=000026Row}
in the case $p=0$. 
\begin{prop}
\label{prop:Bav=000026Row 4.12} Suppose $A$ is semisimple, $p\neq2,3$
and $k\ge0$. Then every Jordan-Lie inner ideal of $A^{(k)}$ is regular. 
\end{prop}

\begin{proof}
This follows from Proposition \ref{cor:JL=00003DeAf} and Lemma \ref{eAf is regular}(i). 
\end{proof}
We will need the following two results which were first proved in
\cite{Bav=000026Row} in the case when $p=0$. One can easily check
that their proofs in \cite{Bav=000026Row} apply to any $p$.
\begin{prop}
\label{Bav=000026Row 4.8}\cite[Proposition 4.8]{Bav=000026Row} Let
$A$ be an associative ring. Then 

(i) $A$ is Von Neumann regular if and only if $\ccR\ccL=\ccR\cap\ccL$
for all left and right ideals $\ccL$ and $\ccR$, respectively, in
$A$;

(ii) every $x$ in $A$ with $x^{2}=0$ is Von Neumann regular if
and only if $\ccR\ccL=\ccR\cap\ccL$ for all left and right ideals
$\ccL$ and $\ccR$, respectively, in $A$ such that $\ccL\ccR=0$.
\end{prop}

\begin{prop}
\label{prop:Bav=000026Row 4.9}\cite[Proposition 4.9]{Bav=000026Row}
Let $B$ be a subspace of $L=A^{(k)}$ ($k\ge0$). Then $B$ is a
regular inner ideal of $L$ if and only if there exist left and right
ideals $\ccL$ and $\ccR$ of $A$, respectively, such that $\ccL\ccR=0$
and 
\[
\ccR\ccL\subseteq B\subseteq\ccR\cap\ccL.
\]
In particular, if $A$ is Von Neumann regular then every regular inner
ideal of $L$ is of the form $B=\ccR\ccL=\ccR\cap\ccL$. 
\end{prop}

Let $\ccL$ be a left ideal of $A$ and let $X$ be a left ideal of
$\bar{A}$. Then $\ccL$ is said to be $X$-minimal if $\bar{\ccL}=X$
and for every left ideal $\ccL'$ of $A$ with $\ccL'\subseteq\ccL$
and $\bar{\ccL}'=X$ one has $\ccL=\ccL'$. We will need the following
theorem from \cite{Bav=000026Mud=000026Shk}. 
\begin{thm}
\cite{Bav=000026Mud=000026Shk}\label{thm:left=00003DAe} Let $A$
be a left Artinian associative ring and let $\ccL$ be a left ideal
of $A$. If $\ccL$ is $\bar{\ccL}$-minimal, then $\ccL=Ae$ for
some idempotent $e\in\ccL$.
\end{thm}

\begin{thm}
\label{thm:Regular B=00003DeAf } Let $B$ be a bar-minimal Jordan-Lie
inner ideal of $L=A^{(k)}$ ($k\ge0$). Then the following holds.

(i) If $B$ is regular then $B=eAf$ for some orthogonal idempotent
pair $(e,f)$ in $A$. 

(ii) Suppose $k=0,1$. Then $B$ is regular if and only if $B=eAf$
for some orthogonal idempotent pair $(e,f)$ in $A$.
\end{thm}

\begin{proof}
(i) Suppose that $B$ is regular. Then by Proposition \ref{prop:Bav=000026Row 4.9},
there are left $\ccL$ and right $\ccR$ ideals of $A$ such that
$\ccL\ccR=0$ and $\ccR\ccL\subseteq B\subseteq\ccR\cap\ccL$. Hence,
$\overline{\ccR\ccL}=\bar{\ccR}\bar{\ccL}\subseteq\bar{B}\subseteq\bar{\ccL}\cap\bar{\ccR}$.
Since $\bar{A}$ is Von Neumann regular (because it is semisimple),
by Proposition \ref{Bav=000026Row 4.8}, $\bar{\ccR}\bar{\ccL}=\bar{B}$.
Let $\ccL'\subseteq\ccL$ (resp. $\ccR'\subseteq\ccR$) be an $\bar{\ccL}$-minimal
left (resp. $\bar{\ccR}$-minimal right) ideal of $A$. Then by Theorem
\ref{thm:left=00003DAe}, $\ccL'=Af$ and $\ccR'=eA$ for some idempotents
$e\in\ccR'$ and $f\in\ccL'$. Note that $fe\in\ccL'\ccR'\subseteq\ccL\ccR=0.$
Put $B'=\ccR'\ccL'\subseteq B$. Then $B'=eAAf=eAf$ (as $eAf=eeAf\subseteq eAAf\subseteq eAf$).
Since $B'^{2}=0$, by Proposition \ref{prop:Bav=000026Row 4.9}, $B'$
is a regular inner ideal of $L$. As $\overline{B'}=\overline{\ccR'}\,\overline{\ccL'}=\bar{\ccR}\bar{\ccL}=\bar{B}$
and $B$ is bar-minimal, $B=B'.$ Thus, $B=eAf$ for some idempotents
$e$ and $f$ in $A$ with $fe=0$. Therefore, by Lemma \ref{lem:eAf cap Z =00003D 0}(iv),
$B=eAf=eAg$ for some idempotent $g$ in $A$ with $ge=eg=0$.

(ii) This follows from (i) and Lemma \ref{eAf is regular}.
\end{proof}

\section{Proof of the main results }

The aim of this section is to prove that bar-minimal Jordan-Lie inner
ideals are generated by idempotents (Theorem \ref{thm:main iff})
and are regular (Corollary \ref{cor:main regular}). As a corollary,
we show that all Jordan-Lie inner ideals split (Corollary \ref{cor:main split}).
Recall that $S$ is a Levi subalgebra of $A$, $L=A^{(k)}=S^{(k)}\oplus N$,
for some $k\ge0$, $N=R\cap L$, and $\bar{B}$ is the image of $B$
in $\bar{L}=L+R/R\cong L/N$. 

First we consider the case when $A$ is $1$-perfect. Then $L=[A,A]$
is a perfect Lie algebra for $p\ne2$ (see Theorem \ref{BaranovLiePerfect}).
The following theorem will be proved in steps. 
\begin{thm}
\label{thm:B=00003DeAf } Let $L=[A,A]$ and let $B$ be a Jordan-Lie
inner ideal of $L$. Suppose that $p\neq2,3$, $A$ is $1$-perfect
and $B$ is bar-minimal. Then the following hold. 

(i) $B$ splits in $A$. 

(ii) $B=eAf$ for some strict orthogonal idempotent pair $(e,f)$
in $A$.

(iii) $B$ is regular. 
\end{thm}

First we will consider the case when $R^{2}=0$. 
\begin{thm}
\label{thm:minimal splits if R^2=00003D0} Let $L=[A,A]$ and let
$B$ be a Jordan-Lie inner ideal of $L$. Suppose that $p\neq2,3$,
$A$ is $1$-perfect, $B$ is bar-minimal and $R^{2}=0$. Then $B$
splits in $A$. 
\end{thm}

Theorem \ref{thm:minimal splits if R^2=00003D0} first appeared in
Rowley's thesis \cite{Rowley} in the case when $p=0$ and we use
some of his ideas below. Unfortunately, his proof is incomplete and
contains some inaccuracies. In particular, the proof of \cite[Proposition 6.12]{Rowley}
is incorrect. We will need the following lemma. 
\begin{lem}
\label{lem:(i) radL  (ii) B } Let $L=[A,A]$ and $Q=[S,S]$. Suppose
that $p\neq2$, $A/R$ is simple, $RA=0$ and $R$ is an irreducible
left $A$-module. Then the following hold. 

(i) $N=R$. 

(ii) Every Jordan-Lie inner ideal of $Q$ is a Jordan-Lie inner ideal
of $L$. 

(iii) Let $G$ be a large subalgebra of $A$ and let $B$ be a Jordan-Lie
inner ideal of $[G,G]$. Then $B$ is a Jordan-Lie inner ideal of
$L$. 
\end{lem}

\begin{proof}
(i) Let $r\in R$. Since $R$ is irreducible as $S$-module, $r=sr$
for some $s\in S$. As $RA=0$, $r=sr=[s,r]\in[S,R]=N$ by Proposition
\ref{prop:A is strongly L=00003D=00005BA,A=00005D}, so $R=N$. 

(ii) This follows from (iii) as $Q=[S,S]$ and $S$ is a large subalgebra
of $A$. 

(iii) Since $G$ is a large subalgebra of $A$, it contains a Levi
subalgebra of $A$. Without loss of generality we can assume $S\subseteq G$.
Let $x\in L$. Since $L=[A,A]\subseteq Q\oplus R$, $x=q+r$ for some
$q\in Q$ and $r\in R$. As $RA=0$, for all $b,b'\in B$ we have
\[
\{b,x,b'\}=bxb'+b'xb=b(q+r)b'+b'(q+r)b=bqb'+b'qb=\{b,q,b'\}\in B,
\]
i.e. $B$ is an inner ideal of $L$, as required. 
\end{proof}
Recall that $A$ is a $1$-perfect finite dimensional associative
algebra, $R$ is the radical of $A$ with $R^{2}=0$ and $S$ is a
Levi subalgebra of $A$, so by Proposition \ref{prop:A is strongly L=00003D=00005BA,A=00005D},
$L=[A,A]$ is a perfect Lie algebra, $Q=[S,S]$ is a quasi Levi subalgebra
of $L$ and $L=Q\oplus N$ is a quasi Levi decomposition of $L$,
where $N=[S,R]$. 
\begin{prop}
\label{prop:B splits in AR=00003D0 =000026 R irreducible} Theorem
\ref{thm:minimal splits if R^2=00003D0} holds if $A/R$ is simple,
$RA=0$ and $R$ is an irreducible left $A$-module. Moreover, $B\subseteq S'$
for some Levi subalgebra $S'$ of $A$. 
\end{prop}

\begin{proof}
By Lemma \ref{lem:(i) radL  (ii) B }, $R$ coincides with the nil-radical
$N$ of $L$. We identify $\bar{A}$ with $S$. Recall that $B$ is
bar-minimal. We are going to prove that there is a Levi subalgebra
$S'$ of $A$ such that $B\subseteq S'$, so $B$ splits in $A$.
Since $S\cong A/R$ is simple, by Lemma \ref{lem:X=00003Dspan=00007Be:1<i<k<l<j<n=00007D},
there is a matrix realization $\ccM_{n}$ of $S$ and integers $1\leq k<l\leq n$
such that $\bar{B}$ is the space spanned by $E=\{e_{st}\mid1\leq s\leq k<l\leq t\leq n\}$
where $\{e_{ij}\mid1\leq i,j\leq n\}$ is the standard basis of $S$
consisting of matrix units. Since $R$ is an irreducible left $S$-module,
it can be identified with the natural $n$-dimensional left $S$-module
$V$. Let $\{e_{1},e_{2},\ldots,e_{n}\}$ be the standard basis of
$V$. Fix $b_{st}^{(1)}\in B$ such that $\overline{b_{st}^{(1)}}=e_{st}$
for all $s$ and $t$. Then $b_{st}^{(1)}=e_{st}+r_{st}$, where $r_{st}\in R$.
Put 
\[
\Lambda_{1}=\{b_{st}^{(1)}=e_{st}+r_{st}:1\leq s\leq k<l\leq t\leq n\}\subseteq B.
\]
 Since $e_{ts}\in L$, by Lemma \ref{lem:Jordan-Lie }, $b_{st}^{(2)}=b_{st}^{(1)}e_{ts}b_{st}^{(1)}\in B$.
Let $r_{st}=\sum_{i=1}^{n}\alpha_{i}^{st}e_{i}$, where $\alpha_{i}^{st}\in\bbF$.
Then 
\[
b_{st}^{(2)}=b_{st}^{(1)}e_{ts}b_{st}^{(1)}=(e_{st}+\sum_{i=1}^{n}\alpha_{i}^{st}e_{i})e_{ts}(e_{st}+\sum_{i=1}^{n}\alpha_{i}^{st}e_{i})=e_{ss}(e_{st}+\sum_{i=1}^{n}\alpha_{i}^{st}e_{i})=e_{st}+\alpha_{s}^{st}e_{s}.
\]
Hence, the set 
\[
\Lambda_{2}=\{b_{st}^{(2)}=e_{st}+\alpha_{s}^{st}e_{s}:1\leq s\leq k<l\leq t\leq n\}\subseteq B.
\]
Put $b_{1t}^{(3)}=b_{1t}^{(2)}=e_{1t}+\alpha_{1}^{1t}e_{1}$ and for
$s>1$ set $b_{st}^{(3)}=\{b_{st}^{(2)},e_{t1},b_{1t}^{(2)}\}$. Then
by Lemma \ref{lem:Jordan-Lie }, $b_{st}^{(3)}\in B$. Since $RA=0$,
for $s>1$ we have 
\begin{eqnarray*}
b_{st}^{(3)} & = & \{b_{st}^{(2)},e_{t1},b_{1t}^{(2)}\}=b_{st}^{(2)}e_{t1}b_{1t}^{(2)}+b_{1t}^{(2)}e_{t1}b_{st}^{(2)}\\
 & = & (e_{st}+\alpha_{s}^{st}e_{s})e_{t1}b_{1t}^{(2)}+(e_{1t}+\alpha_{1}^{1t}e_{t})e_{t1}b_{st}^{(2)}\\
 & = & e_{s1}(e_{1t}+\alpha_{1}^{1t}e_{1})+e_{11}(e_{st}+\alpha_{s}^{st}e_{s})\\
 & = & e_{st}+\alpha_{1}^{1t}e_{s}.
\end{eqnarray*}
Denote $\beta_{t}=\alpha_{1}^{1t}$ for all $t$. Then $b_{st}^{(3)}=e_{st}+\beta_{t}e_{s}\in B$
for all $s$ and $t$. Thus 
\[
\Lambda_{3}=\{b_{st}^{(3)}=e_{st}+\beta_{t}e_{s}:1\leq s\leq k<l\leq t\leq n\}\subseteq B.
\]
 Let $q=\sum_{j=l}^{n}\beta_{j}e_{j}\in R$. Then $q^{2}\in R^{2}=0$.
Define the  special inner automorphism $\varphi:A\rightarrow A$ by
$\varphi(a)=(1+q)a(1-q)$ for all $a\in A$. Since $RA=0$, by applying
$\varphi$ to all $b_{st}^{(3)}\in\Lambda_{3}$ we obtain 
\begin{eqnarray*}
\varphi(b_{st}^{(3)}) & = & (1+\sum_{j=l}^{n}\beta_{j}e_{j})(e_{st}+\beta_{t}e_{s})(1-q)\\
 & = & (e_{st}+\beta_{t}e_{s})(1-\sum_{j=l}^{n}\beta_{j}e_{j})=e_{st}+\beta_{t}e_{s}-\beta_{t}e_{s}=e_{st}\in\varphi(B)
\end{eqnarray*}
Therefore, 
\[
E=\{e_{st}\mid1\leq s\leq k<l\leq t\leq n\}\subseteq\varphi(B)\cap S.
\]
Note that $\varphi(r)=r$ for all $r\in R$. Hence, $\varphi(B)=\varphi(B)_{S}\oplus\varphi(B)_{R}$,
where $\varphi(B)_{S}=\varphi(B)\cap S$ and $\varphi(B)_{R}=\varphi(B)\cap R$.
By changing the Levi subalgebra $S$ of $A$ to $S'=\varphi^{-1}(S)$
we obtain $B=B_{S'}\oplus B_{R}$, where $B_{S'}=B\cap S'$ and $B_{R}=B\cap R$.
Therefore, $B$ splits in $A$. 

It remains to show that $B\subseteq S'$. Let $P=[B_{S'},[B_{S'},S'^{(1)}]]\subseteq S'^{(1)}\cap B$.
Since $\bar{A}$ is semisimple and $\bar{B}_{S'}=\bar{B}$, we get
that 
\[
\bar{P}=[\bar{B}_{S'},[\bar{B}_{S'},\bar{S}'^{(1)}]]=[\bar{B},[\bar{B},\bar{A}^{(1)}]]=\bar{B}.
\]
Note that $B'=B\cap S'{}^{(1)}$ is a Jordan-Lie inner ideal of $S'^{(1)}$.
As $P\subseteq B'$, we have $\bar{B}=\bar{P}=\bar{B}'$. By Lemma
\ref{lem:(i) radL  (ii) B }, $B'$ is a Jordan-Lie inner ideal of
$L$. Since $\bar{B}'=\bar{B}$ and $B$ is bar-minimal, we have $B=B'\subseteq S'$,
as required. 
\end{proof}
\begin{prop}
\label{prop:RA=00003D0} Theorem \ref{thm:minimal splits if R^2=00003D0}
holds if $A/R$ is simple and $RA=0$. 
\end{prop}

\begin{proof}
Since $A$ is $1$-perfect, $SR=R$, so $R$ as a left $S$-module
is a direct sum of copies of the natural left $S$-module $V$. The
proof is by induction on the length $\ell(R)$ of the left $S$-module
$R$, the case $\ell(R)=1$ being clear by Proposition \ref{prop:B splits in AR=00003D0 =000026 R irreducible}.
Suppose that $\ell(R)>1$. Consider any maximal submodule $T$ of
$R$. Then $\ell(T)=\ell(R)-1$ and $T$ is an ideal of $A$. Let
$\tilde{\quad}:A\rightarrow A/T$ be the natural epimorphism of $A$
onto $\tilde{A}=A/T$. Denote by $\tilde{R}$ and $\tilde{B}$ the
images of $R$ and $B$, respectively, in $\tilde{A}$. Since $\ell(\tilde{R})=1$,
by Proposition \ref{prop:B splits in AR=00003D0 =000026 R irreducible},
$\tilde{B}$ is contained in a Levi subalgebra of $\tilde{A}$. Therefore,
$B\subseteq S_{1}\oplus T$ for some Levi subalgebra $S_{1}$ of $A$.
Put $G=S_{1}\oplus T$. Then $G$ is clearly $1$-perfect (i.e. $G=\ccP_{1}(G)$),
$\rad(G)=T$, $G=S_{1}\oplus T$ is a Levi decomposition of $G$ and
$C=B\cap G^{(1)}$ is a Jordan-Lie inner ideal of $G^{(1)}=\ccP_{1}(G)^{(1)}$.
Put $P=[B,[B,G^{(1)}]]\subseteq C$. Then 
\[
\bar{P}=[\bar{B},[\bar{B},\bar{G}^{(1)}]]=[\bar{B},[\bar{B},\bar{A}^{(1)}]]=\bar{B},
\]
so $\bar{C}=\bar{B}$. Let $C'$ be any $\bar{C}$-minimal inner ideal
of $G^{(1)}$ contained in $C$. Since $G$ is $1$-perfect and $\ell(T)<\ell(R)$,
by the inductive hypothesis, $C'$ splits in $G$. Since $C'\subseteq C\subseteq B$
and $\bar{C}'=\bar{C}=\bar{B}$, by Proposition \ref{lem:Large is splits},
$B$ splits in $A$. 
\end{proof}
\begin{prop}
\label{prop:AR=00003D0} Theorem \ref{thm:minimal splits if R^2=00003D0}
holds if $A/R$ is simple and $AR=0$. 
\end{prop}

\begin{proof}
The proof is similar to that of Proposition \ref{prop:RA=00003D0}. 
\end{proof}
\begin{prop}
\label{prop: R is S-bimodule} Theorem \ref{thm:minimal splits if R^2=00003D0}
holds if $A/R$ is simple and $R$ is isomorphic to the natural $A/R$-bimodule
$A/R$ with respect to the right and left multiplication.
\end{prop}

\begin{proof}
Recall that $B$ is a Jordan-Lie inner ideal of $L=[A,A]$ such that
$B$ is bar-minimal. As in the proof of Proposition \ref{prop:B splits in AR=00003D0 =000026 R irreducible},
we fix standard bases $\{e_{ij}\mid1\leq i,j\leq n\}$ and $\{f_{ij}\mid1\leq i,j\leq n\}$
of $S$ and $R$, respectively, consisting of matrix units, such that
the action of $S$ on $R$ corresponds to matrix multiplication and
$\bar{B}$ is the space spanned by $E=\{e_{st}\mid1\leq s\leq k<l\leq t\leq n\}\subseteq S$.
We identify $\bar{A}$ with $S$. We are going to prove that there
is a Levi subalgebra $S'$ of $A$ such that $B=B_{S'}\oplus B_{R}$,
where $B_{S'}=B\cap S'$ and $B_{R}=B\cap R$ . Put 
\[
R_{0}=span\{f_{st}\mid1\leq s\leq k<l\leq t\leq n\}\subseteq N.
\]

\noun{Claim 1:} $R_{0}\subseteq B$. Fix any $b_{st}\in B$ such that
$\bar{b}_{st}=e_{st}$. Then $b_{st}=e_{st}+r_{st}$, with $r_{st}\in N$.
By Lemma \ref{lem:Jordan-Lie }, $b_{st}f_{ts}b_{st}\in B$. Since
$R^{2}=0$, we have
\[
b_{st}f_{ts}b_{st}=(e_{st}+r_{st})f_{ts}(e_{st}+r_{st})=f_{ss}(e_{st}+r_{st})=f_{st}.
\]
 Therefore, $f_{st}\in B$ for all $s$ and $t$ as required. 

\noun{Claim 2: }For every $b_{st}=e_{st}+\sum_{i,j=1}^{n}\alpha_{ij}^{st}f_{ij}\in B$
we have 
\[
\theta(b_{st})=e_{st}+\sum_{i>k}\alpha_{it}^{st}f_{it}+\sum_{j<l}\alpha_{sj}^{st}f_{sj}\in B.
\]
 Since $b_{st}\in B$, by Lemma \ref{lem:Jordan-Lie }, $b_{st}e_{ts}b_{st}\in B$.
We have 
\begin{eqnarray*}
b_{st}e_{ts}b_{st} & = & (e_{st}+\sum_{i,j}^{n}\alpha_{ij}^{st}f_{ij})e_{ts}b_{st}=(e_{ss}+\sum_{i}^{n}\alpha_{it}^{st}f_{is})(e_{st}+\sum_{i,j}^{n}\alpha_{ij}^{st}f_{ij})\\
 & = & e_{st}+\sum_{i}^{n}\alpha_{it}^{st}f_{it}+\sum_{j}^{n}\alpha_{sj}^{st}f_{sj}=\theta(b_{st})+\sum_{i=1}^{k}\alpha_{it}^{st}f_{it}+\sum_{j=l}^{n}\alpha_{sj}^{st}f_{sj}.
\end{eqnarray*}
Since $\sum_{i=1}^{k}\alpha_{it}^{st}f_{it}+\sum_{j=l}^{n}\alpha_{sj}^{st}f_{sj}\in R_{0}\subseteq B$
and $b_{st}e_{ts}b_{st}\in B$, we have $\theta(b_{st})\in B$ as
required. 

By claim 2, there are some $\alpha_{ij}^{st}\in\bbF$ such that 
\[
b_{st}=e_{st}+\sum_{i>k}\alpha_{it}^{st}f_{it}+\sum_{j<l}\alpha_{sj}^{st}f_{sj}\in B,
\]
for all $1\leq s\leq k<l\leq t\leq n$. 

(1) Define the special inner automorphism $\varphi_{1}:A\rightarrow A$
by $\varphi_{1}(a)=(1+q_{1})a(1-q_{1})$ for all $a\in A$, where
\[
q_{1}=\sum_{j<l}\alpha_{1j}^{1n}f_{nj}-\sum_{i>k}\alpha_{in}^{1n}f_{i1}\in R.
\]
Put $B_{1}=\varphi_{1}(B)$. Set $b_{st}^{(1)}=\varphi_{1}(b_{st})$
for all $s$ and $t$. Then 
\begin{eqnarray*}
b_{1n}^{(1)} & = & (1+q_{1})b_{1n}(1-q_{1})\\
 & = & (1+\sum_{j<l}\alpha_{1j}^{1n}f_{nj}-\sum_{i>k}\alpha_{in}^{1n}f_{i1})(e_{1n}+\sum_{i>k}\alpha_{in}^{1n}f_{in}+\sum_{j<l}\alpha_{1j}^{1n}f_{1j})(1-q_{1})\\
 & = & (e_{1n}+\sum_{i>k}\alpha_{in}^{1n}f_{in}+\sum_{j<l}\alpha_{1j}^{1n}f_{1j}+\alpha_{11}^{1n}f_{nn}-\sum_{i>k}\alpha_{in}^{1n}f_{in})(1-q_{1})\\
 & = & (e_{1n}+\sum_{j<l}\alpha_{1j}^{1n}f_{1j}+\alpha_{11}^{1n}f_{nn})(1-\sum_{j<l}\alpha_{1j}^{1n}f_{nj}+\sum_{i>k}\alpha_{in}^{1n}f_{i1})\\
 & = & e_{1n}+\sum_{j<l}\alpha_{1j}^{1n}f_{1j}+\alpha_{11}^{1n}f_{nn}-\sum_{j<l}\alpha_{1j}^{1n}f_{1j}+\alpha_{nn}^{1n}f_{11}\\
 & = & e_{1n}+\alpha_{11}^{1n}f_{nn}+\alpha_{nn}^{1n}f_{11}.
\end{eqnarray*}
 Since $(B_{1})^{2}=0$, we have 
\begin{eqnarray*}
0 & = & (b_{1n}^{(1)})^{2}=(e_{1n}+\alpha_{11}^{1n}f_{nn}+\alpha_{nn}^{1n}f_{11})(e_{1n}+\alpha_{11}^{1n}f_{nn}+\alpha_{nn}^{1n}f_{11})\\
 & = & \alpha_{11}^{1n}f_{1n}+\alpha_{nn}^{1n}f_{1n}=(\alpha_{11}^{1n}+\alpha_{nn}^{1n})f_{1n}.
\end{eqnarray*}
Thus, $\alpha_{11}^{1n}=-\alpha_{nn}^{1n}$. Put $\alpha=\alpha_{11}^{1n}$.
Then 
\begin{equation}
b_{1n}^{(1)}=e_{1n}+\alpha f_{11}-\alpha f_{nn}\in B_{1}\label{eq: s=00003D1 =000026 d=00003Dn}
\end{equation}

(2) Consider the special inner automorphism $\varphi_{2}:A\rightarrow A$
defined by $\varphi_{2}(a)=(1+\alpha f_{n1})a(1-\alpha f_{n1})$ for
all $a\in A$. Put $B_{2}=\varphi_{2}(B_{1})$. Then by applying $\varphi_{2}$
to (\ref{eq: s=00003D1 =000026 d=00003Dn}), we obtain 
\begin{eqnarray*}
b_{1n}^{(2)} & = & \varphi_{2}(b_{1n}^{(1)})=(1+\alpha f_{n1})(e_{1n}+\alpha f_{11}-\alpha f_{nn})(1-\alpha f_{n1})\\
 & = & (e_{1n}+\alpha f_{11}-\alpha f_{nn}+\alpha f_{nn})(1-\alpha f_{n1})=e_{1n}+\alpha f_{11}-\alpha f_{11}=e_{1n}\in B_{2}.
\end{eqnarray*}
Put $b_{st}^{(2)}=\theta(\varphi_{2}(b_{st}^{(1)}))\in B_{2}$ for
all $s$ and $t$. Then $b_{st}^{(2)}=e_{st}+\underset{i>k}{\sum}\beta_{it}^{st}f_{it}+\underset{j<l}{\sum}\beta_{sj}^{st}f_{sj}$,
where $\beta_{ij}^{st}\in\bbF$. 

Put $b_{1n}^{(3)}=b_{1n}^{(2)}=e_{1n}$, $b_{st}^{(3)}=b_{st}^{(2)}$
for $t\neq n$ and $b_{sn}^{(3)}=\{b_{sn}^{(2)},e_{n1},e_{1n}\}$
for $s\neq1$. Then by Lemma \ref{lem:Jordan-Lie }, $b_{sn}^{(3)}\in B_{2}$
for all $s$ and $t$. Thus, for $s\neq1$ we have 
\begin{eqnarray}
b_{sn}^{(3)} & = & \{b_{sn}^{(2)},e_{n1},e_{1n}\}=b_{sn}^{(2)}e_{n1}e_{1n}+e_{1n}e_{n1}b_{sn}^{(2)}=b_{sn}^{(2)}e_{nn}+e_{11}b_{sn}^{(2)}\nonumber \\
 & = & (e_{sn}+\sum_{i>k}\beta_{in}^{sn}f_{in}+\sum_{j<l}\beta_{sj}^{sn}f_{sj})e_{nn}+e_{11}(e_{sn}+\sum_{i>k}\beta_{in}^{sn}f_{in}+\sum_{j<l}\beta_{sj}^{sn}f_{sj})\nonumber \\
 & = & e_{sn}+\sum_{i>k}\beta_{in}^{sn}f_{in}\in B_{2}.\label{eq:eq:b^3 in B2}
\end{eqnarray}
Note that $b_{1n}^{(3)}=e_{1n}$ is also of the shape (\ref{eq:eq:b^3 in B2})
with all $\beta_{in}^{1n}=0$. 

(3) Consider the special inner automorphism $\varphi_{3}:A\rightarrow A$
defined by $\varphi_{3}(a)=(1+q_{3})a(1-q_{3})$ for $a\in A$, where
\[
q_{3}=-\sum_{i>k}\sum_{j=2}^{k}\beta_{in}^{jn}f_{ij}.
\]
Put $B_{3}=\varphi_{3}(B_{2})$ and $b_{st}^{(4)}=\varphi_{3}(b_{st}^{(3)})\in B_{3}$.
By applying $\varphi_{3}$ to $b_{sn}^{(3)}$ in (\ref{eq:eq:b^3 in B2})
(for all $s$), we obtain 
\begin{eqnarray}
b_{sn}^{(4)} & = & \varphi_{3}(b_{sn}^{(3)})=(1+q_{3})b_{sn}^{(3)}(1-q_{3})=(1-\sum_{i>k}\sum_{j=2}^{k}\beta_{in}^{jn}f_{ij})(e_{sn}+\sum_{i>k}\beta_{in}^{sn}f_{in})(1-q_{3})\nonumber \\
 & = & (e_{sn}+\sum_{i>k}\beta_{in}^{sn}f_{in}-\sum_{i>k}\beta_{in}^{sn}f_{in})(1+\sum_{i>k}\sum_{j=2}^{k}\beta_{in}^{jn}f_{ij})=e_{sn}+\sum_{j=2}^{k}\beta_{nn}^{jn}f_{sj}\in B_{3}\label{eq:b^(4)}
\end{eqnarray}
 Since $(B_{3})^{2}=0$, for all $1\leq s,p\leq k$ we have 
\[
0=b_{sn}^{(4)}b_{rn}^{(4)}=(e_{sn}+\sum_{j=2}^{k}\beta_{nn}^{jn}f_{sj})(e_{rn}+\sum_{j=2}^{n}\beta_{nn}^{jn}f_{rj})=\beta_{nn}^{rn}f_{sn}.
\]
Thus, $\beta_{nn}^{rn}=0$ for all $1\leq r\leq k$. Substituting
in (\ref{eq:b^(4)}) we obtain 
\[
b_{sn}^{(4)}=e_{sn}\in B_{3}\text{ for all }1\le s\le k.
\]

Put $b_{sn}^{(5)}=b_{sn}^{(4)}=e_{sn}$ and $b_{st}^{(5)}=\theta(b_{st}^{(4)})\in B_{3}$
for $t\neq n$. Then for $t\neq n$ we have 
\[
b_{st}^{(5)}=e_{st}+\sum_{i>k}\gamma_{it}^{st}f_{it}+\sum_{j<l}\gamma_{sj}^{st}f_{sj}\text{ for some }\gamma_{ij}^{st}\in\bbF.
\]

Put $b_{sn}^{(6)}=b_{sn}^{(5)}=e_{sn}$ and $b_{st}^{(6)}=\{e_{sn},e_{n1},b_{1t}^{(5)}\}$
for all $t\neq n$. Then by Lemma \ref{lem:Jordan-Lie }, $b_{st}^{(6)}\in B_{3}$.
Thus, for $t\neq n$ we have 
\begin{eqnarray*}
b_{st}^{(6)} & = & \{e_{sn},e_{n1},b_{1t}^{(5)}\}=e_{sn}e_{n1}b_{1t}^{(5)}+b_{1t}^{(5)}e_{n1}e_{sn}=e_{s1}b_{1t}^{(5)}+0\\
 & = & e_{s1}(e_{1t}+\sum_{i>k}\gamma_{it}^{1t}f_{it}+\sum_{j<l}\gamma_{1j}^{1t}f_{1j})=e_{st}+\sum_{j<l}\gamma_{1j}^{1t}f_{sj}.
\end{eqnarray*}

(4) We define the final special inner automorphism $\varphi_{4}:A\rightarrow A$
by $\varphi_{4}(a)=(1+q_{4})a(1-q_{4})$ for $a\in A$, where 
\[
q_{4}=\sum_{i=l}^{n-1}\sum_{j<l}\gamma_{1j}^{1i}f_{ij}.
\]
Put $B_{4}=\varphi_{4}(B_{3})$ and $b_{st}^{(7)}=\varphi_{4}(b_{st}^{(6)})\in B_{4}$
for all $s$ and $t$. Then 
\begin{eqnarray*}
b_{sn}^{(7)} & = & (\varphi_{4}(b_{sn}^{(6)}))=(1+q_{4})b_{sn}^{(6)}(1-q_{4})=(1+\sum_{i=l}^{n-1}\sum_{j<l}\gamma_{1j}^{1i}f_{ij})e_{sn}(1-q_{4})\\
 & = & (e_{sn}+\sum_{i=l}^{n-1}\gamma_{1s}^{1i}f_{in})(1-\sum_{i=l}^{n-1}\sum_{j<l}\gamma_{1j}^{1i}f_{ij})=e_{sn}+\sum_{i=l}^{n-1}\gamma_{1s}^{1i}f_{in}\in B_{4}
\end{eqnarray*}
and (for all $t<n$) 
\begin{eqnarray*}
b_{st}^{(7)} & = & (\varphi_{4}(b_{st}^{(6)}))=(1+q_{4})b_{st}^{(6)}(1-q_{4})\\
 & = & (1+\sum_{i=l}^{n-1}\sum_{j<l}\gamma_{1j}^{1i}f_{ij})(e_{st}+\sum_{j<l}\gamma_{1j}^{1t}f_{sj})(1-q_{4})\\
 & = & (e_{st}+\sum_{j<l}\gamma_{1j}^{1t}f_{sj}+\sum_{i=l}^{n-1}\gamma_{1s}^{1i}f_{it})(1-\sum_{i=l}^{n-1}\sum_{j<l}\gamma_{1j}^{1i}f_{ij})\\
 & = & e_{st}-\sum_{j<l}\gamma_{1j}^{1t}f_{sj}+\sum_{j<l}\gamma_{1j}^{1t}f_{sj}+\sum_{i=l}^{n-1}\gamma_{1s}^{1i}f_{it}\\
 & = & e_{st}+\sum_{i=l}^{n-1}\gamma_{1s}^{1i}f_{it}\in B_{4},
\end{eqnarray*}
so 
\begin{equation}
b_{st}^{(7)}=e_{st}+\sum_{i=l}^{n-1}\gamma_{1s}^{1i}f_{it}\hspace*{1em}\text{for all}\hspace*{1em}1\leq s\leq k<l\leq t\leq n.\label{eq:b^(7)}
\end{equation}
Since $(B_{4})^{2}=0$, for all $l\le q\le n-1$, we have 
\[
0=b_{sq}^{(7)}b_{st}^{(7)}=(e_{sq}+\sum_{i=l}^{n-1}\gamma_{1s}^{1i}f_{iq})(e_{st}+\sum_{i=l}^{n-1}\gamma_{1s}^{1i}f_{it})=\gamma_{1s}^{1q}f_{st}.
\]
Thus, $\gamma_{1s}^{1i}=0$ for all $1\leq i\leq n-1$. Substituting
in (\ref{eq:b^(7)}) we obtain $b_{st}^{(7)}=e_{st}$ for all $s$
and $t$. Thus, 
\[
E=\{e_{st}\mid1\leq s\leq k<l\leq t\leq n\}\subseteq B_{4}\cap S.
\]
Denote by $\varphi$ the automorphism $\varphi_{4}\circ\varphi_{3}\circ\varphi_{2}\circ\varphi_{1}$
of $A$ and $L$. Then we have $E\subseteq\varphi(B)\cap S$. Note
that $\varphi_{i}(R_{0})=R_{0}$ for all $i=1,2,3,4$ (because $R^{2}=0$).
Hence, $\varphi(B)=\varphi(B)_{S}\oplus\varphi(B)_{R}$, where $\varphi(B)_{S}=\varphi(B)\cap S$
and $\varphi(B)_{R}=\varphi(B)\cap R=B_{R}$. Now, by changing the
Levi subalgebra $S$ to $S'=\varphi^{-1}(S)$ we obtain $B=B_{S}\oplus B_{R}$,
where $B_{S'}=B\cap S'$ and $B_{R}=B\cap R$. 
\end{proof}
\begin{prop}
\label{prop: A=00003DS+R, S=00003DS1+S2} Theorem \ref{thm:minimal splits if R^2=00003D0}
holds if $A/R\cong S_{1}\oplus S_{2}$, where $S_{1}\cong M_{n_{1}}(\bbF)$,
$S_{2}\cong M_{n_{2}}(\bbF)$ and $R\cong M_{n_{1}n_{2}}(\bbF)$ as
an $S_{1}$-$S_{2}$-bimodule such that $RS_{1}=S_{2}R=0$. 
\end{prop}

\begin{proof}
Recall that $B$ is a Jordan-Lie inner ideal of $L=[A,A]$ such that
$B$ is bar-minimal. We identify $\bar{A}$ with $S$. By Lemma \ref{lem:B=00003DB1+...},
$\bar{B}=X_{1}\oplus X_{2}$, where $X_{i}=\bar{B}\cap S_{i}$ are
Jordan-Lie inner ideals of $S_{i}^{(1)}$. As in the proof of Proposition
\ref{prop:RA=00003D0}, we fix standard bases $\{e_{ij}\mid1\leq i,j\leq n_{1}\}$,
$\{g_{ij}\mid1\leq i,j\leq n_{2}\}$ and $\{f_{ij}\mid1\leq i\leq n_{1},\ 1\leq j\leq n_{2}\}$
of $S_{1}$, $S_{2}$ and $R$, respectively, consisting of matrix
units, such that the action of $S_{1}$ and of $S_{2}$ on $R$ corresponds
to matrix multiplication and $X_{i}=span\{E_{i}\}$, where 
\[
E_{1}=\{e_{st}\mid1\leq s\leq k_{1}<l_{1}\leq t\leq n_{1}\}\subseteq S_{1},
\]
\[
E_{2}=\{g_{rq}\mid1\leq r\leq k_{2}<l_{2}\leq q\leq n_{2}\}\subseteq S_{2}.
\]
Put $R_{0}=span\{f_{sq}\mid1\leq s\leq k_{1},\ l_{2}\leq q\leq n_{2}\}\subseteq N.$

\noun{Claim 1:} $R_{0}\subseteq B$. Fix any $b_{st},c_{rq}\in B$
such that $\bar{b}_{st}=e_{st}$ and $\bar{c}_{rq}=g_{rq}$. Then
$b_{st}=e_{st}+r_{st}$ and $c_{rq}=g_{rq}+r'_{rq}$, with $r_{st},r'_{rq}\in N$.
By Lemma \ref{lem:Jordan-Lie }, $\{b_{st},f_{tr},c_{rq}\}\in B$.
Since $R^{2}=0$ and $S_{2}R=RS_{1}=0$, we have 
\[
\{b_{st},f_{tr},c_{rq}\}=b_{st}f_{tr}c_{rq}+c_{rq}f_{tr}b_{st}=b_{st}f_{tr}c_{rq}=(e_{st}+r_{st})f_{tr}(g_{rq}+r'_{rq})=f_{sr}(g_{rq}+r'_{rq})=f_{sq}.
\]
 Therefore, $f_{sq}\in B$ for all $1\leq s\leq k_{1}$ and $l_{2}\leq q\leq n_{2}$
as required. 

\noun{Claim 2: }For every $b_{st}=e_{st}+\sum_{i=1}^{n_{1}}\sum_{j=1}^{n_{2}}\alpha_{ij}^{st}f_{ij}\in B$
we have 
\[
\theta(b_{st})=e_{st}+\sum_{j<l_{2}}\alpha_{sj}^{st}f_{sj}\in B
\]
 Since $b_{st}\in B$, by Lemma \ref{lem:Jordan-Lie }, $b_{st}e_{ts}b_{st}\in B$.
Since $RS_{1}=0$ and $R^{2}=0$, we have 
\begin{eqnarray*}
b_{st}e_{ts}b_{st} & = & (e_{st}+\sum_{i=1}^{n_{1}}\sum_{j=1}^{n_{2}}\alpha_{ij}^{st}f_{ij})e_{ts}b_{st}=e_{ss}(e_{st}+\sum_{i=1}^{n_{1}}\sum_{j=1}^{n_{2}}\alpha_{ij}^{st}f_{ij})\\
 & = & e_{st}+\sum_{j=1}^{n_{2}}\alpha_{sj}^{st}f_{sj}=\theta(b_{st})+\sum_{j=l_{2}}^{n_{2}}\alpha_{sj}^{st}f_{sj}.
\end{eqnarray*}
Since $\sum_{j=l_{2}}^{n_{2}}\alpha_{sj}^{st}f_{sj}\in R_{0}\subseteq B$
and $b_{st}e_{ts}b_{st}\in B$, we have $\theta(b_{st})\in B$ as
required.

Put $A_{2}=S_{2}\oplus R$ and $L_{2}=[A_{2},A_{2}]$. Denote $B_{2}=B\cap L_{2}$.
By Lemma \ref{lem:Bav=000026Row 2.16}, $B_{2}$ is an inner ideal
of $L_{2}$. Moreover, $B_{2}$ is a Jordan-Lie inner ideal as $(B_{2})^{2}=0$.
Note that $\bar{B}_{2}=X_{2}$ (because $B_{2}$ contains the preimage
of $X_{2}$ in $B$). By Lemma \ref{lem:B=00003DB1+B2 if L=00003DL1+L2},
$B_{2}$ is $X_{2}$-minimal. Thus, $B_{2}$ satisfies the conditions
of Proposition \ref{prop:AR=00003D0}. Hence, $B_{2}$ splits. Thus,
there is a special inner automorphism $\varphi_{2}:A\rightarrow A$
such that $E_{2}\subseteq\varphi_{2}(B_{2})\subseteq\varphi_{2}(B)$.
We will deal with the inner ideal $\varphi_{2}(B)$ of $L$. Note
that $\overline{\varphi_{2}(B)}=\bar{B}=X$ and $E_{2}\subseteq\varphi_{2}(B)$.
Our aim is to modify $\varphi_{2}(B)$ in such a way that it contains
both $E_{1}$ and $E_{2}$. 

Put $b_{st}^{(1)}=\theta(\varphi_{2}(b_{st}))\in\varphi_{2}(B)$ for
all $1\leq s\leq k_{1}<l_{1}\leq t\leq n_{1}$. Then 
\[
b_{st}^{(1)}=e_{st}+\sum_{j<l_{2}}\alpha_{sj}^{st}f_{sj}
\]
for all $s$ and $t$. Put $b_{1t}^{(2)}=b_{1t}^{(1)}=e_{1t}+\sum_{j<l_{2}}\alpha_{1j}^{1t}f_{1j}$
and for $s>1$ set $b_{st}^{(2)}=\{b_{st}^{(1)},e_{t1},b_{1t}^{(1)}\}$.
Then by Lemma \ref{lem:Jordan-Lie }, $b_{st}^{(2)}\in\varphi_{2}(B)$.
Since $RS_{1}=0$, for $s>1$ we have
\begin{eqnarray*}
b_{st}^{(2)} & = & \{b_{st}^{(1)},e_{t1},b_{1t}^{(1)}\}=b_{st}^{(1)}e_{t1}b_{1t}^{(1)}+b_{1t}^{(1)}e_{t1}b_{st}^{(1)}\\
 & = & (e_{st}+\sum_{j<l_{2}}\alpha_{sj}^{st}f_{sj})e_{t1}b_{1t}^{(1)}+b_{1t}^{(1)}e_{t1}(e_{st}+\sum_{j<l_{2}}\alpha_{sj}^{st}f_{sj})\\
 & = & e_{s1}(e_{1t}+\sum_{j<l_{2}}\alpha_{1j}^{1t}f_{1j})+0=e_{st}+\sum_{j<l_{2}}\alpha_{1j}^{1t}f_{sj}.
\end{eqnarray*}
Thus, for all $s$ and $t$ we have 
\begin{equation}
b_{st}^{(2)}=e_{st}+\sum_{j<l_{2}}\alpha_{1j}^{1t}f_{sj}.\label{eq:b^3 in Mnm}
\end{equation}
Consider the special inner automorphism $\varphi:A\rightarrow A$
defined by $\varphi(a)=(1+q)a(1-q)$ for all $a\in A$, where 
\[
q=\sum_{i=l_{1}}^{n_{1}}\sum_{j<l_{2}}\alpha_{1j}^{1i}f_{ij}.
\]
Since $RS_{1}=0$ and $R^{2}=0$, by applying $\varphi$ to (\ref{eq:b^3 in Mnm})
we obtain 
\begin{eqnarray*}
\varphi(b_{st}^{(2)}) & = & (1+q)b_{st}^{(2)}(1-q)=b_{st}^{(2)}(1-q)=(e_{st}+\sum_{j<l_{2}}\alpha_{1j}^{1t}f_{sj})(1-\sum_{i=l_{1}}^{n_{1}}\sum_{j<l_{2}}\alpha_{1j}^{1i}f_{ij})\\
 & = & e_{st}+\sum_{j<l_{2}}\alpha_{1j}^{1t}f_{sj}-\sum_{j<l_{2}}\alpha_{1j}^{1t}f_{sj}=e_{st}\in\varphi(\varphi_{2}(B)).
\end{eqnarray*}
Thus, $e_{st}\in\varphi(\varphi_{2}(B))$ for all $1\leq s\leq k_{1}<l_{1}\leq t\leq n_{1}$.
Now, by applying $\varphi$ to $g_{rq}\in X_{2}\subseteq\varphi_{2}(B)$
and using $S_{2}R=0$, we obtain 
\begin{eqnarray*}
\varphi(g_{rq}) & = & (1+q)g_{rq}(1-q)=(1+q)g_{rq}=(1+\sum_{i=l_{1}}^{n_{1}}\sum_{j<l_{2}}\alpha_{1j}^{1i}f_{ij})g_{rq}\\
 & = & g_{rq}+\sum_{i=l_{1}}^{n_{1}}\alpha_{1r}^{1i}f_{iq}\in\varphi(\varphi_{2}(B)).
\end{eqnarray*}
Since $(\varphi(\varphi_{2}(B)))^{2}=0$ and both $e_{st}$ and $\varphi(g_{rq})$
are in $\varphi(\varphi_{2}(B))$, we have 
\[
0=e_{st}\varphi(g_{rq})=e_{st}(g_{rq}+\sum_{i=l_{1}}^{n_{1}}\alpha_{1r}^{1i}f_{iq})=\alpha_{1r}^{1t}f_{sq}.
\]
Hence, $\alpha_{1r}^{1t}=0$ for all $1\leq r\leq k_{2}$ and all
$l_{1}\leq t\leq n_{1}$. Thus, $\varphi(g_{rq})=g_{rq}\in\varphi(\varphi_{2}(B))$
for all $r$ and $q$. Therefore, 

\[
E_{1}=\{e_{st}:1\leq s\leq k_{1}<l_{1}\leq t\leq n_{1}\}\subseteq\varphi(\varphi_{2}(B))\cap S
\]
and 
\[
E_{2}=\{g_{rq}:1\leq r\leq k_{2}<l_{2}\leq q\leq n_{2}\}\subseteq\varphi(\varphi_{2}(B))\cap S.
\]
Put $E=E_{1}\cup E_{2}\subseteq\varphi(\varphi_{2}(B))\cap S$. Since
$R^{2}=0$, one can easily check that $\varphi(\varphi_{2}(R_{0}))=R_{0}$.
By changing the Levi subalgebra $S$ to $S'=\varphi^{-1}(\varphi_{2}^{-1}(S))$
we prove that $B$ splits in $A$. 
\end{proof}
We will need the following result.
\begin{lem}
\label{lem:natural bi-module} Let $S$ be a semisimple finite dimensional
associative algebra and let $\{S_{i}\mid i\in I\}$ be the set of
its simple components. Suppose that $M$ is an $S$-bimodule. Then
$M$ is a direct sum of copies of $U_{ij}$, for $i,j\in I\cup\{0\}$,
where $U_{00}$ is the trivial $1$-dimensional $S$-bimodule, $U_{i0}$
is the natural left $S_{i}$-module with $U_{i0}S=0$, $U_{0j}$ is
the natural right $S$-module with $SU_{0j}=0$ and $U_{ij}$ is the
natural $S_{i}$-$S_{j}$-bimodule for $i,j>0$.
\end{lem}

\begin{proof}
Let $\hat{S}=S+\bbF1_{\hat{S}}$, where $1_{\hat{S}}$ is the unity
of $\hat{S}$. Then $\hat{S}$ is a unital algebra. Set $1_{\hat{S}}m=m1_{\hat{S}}=m$
for all $m\in M$. Then $M$ is a unital $\hat{S}$-bimodule. Note
that $\hat{S}=\bigoplus_{i\in I\cup\{0\}}S_{i}$, where $S_{0}=\bbF(1_{\hat{S}}-1_{S})$
is a $1$-dimensional simple component of $\hat{S}$. Thus, as a unital
$\hat{S}$-bimodule $M$ is a direct sum of copies of the natural
$S_{i}$-$S_{j}$-bimodules $U_{ij}$ such that $U_{ij}=S_{i}US_{j}$,
for all $i$ and $j$. It remains to note that $U_{i0}S=0$ and $SU_{0j}=0$. 
\end{proof}
Now, we are ready to prove Theorem \ref{thm:minimal splits if R^2=00003D0}.
\begin{proof}[Proof of Theorem \ref{thm:minimal splits if R^2=00003D0}]
 Recall that $A$ is $1$-perfect with $R^{2}=0$, $p\neq2,3$ and
$B$ is a bar-minimal Jordan-Lie inner ideal of $L=[A,A]$. Let $\{S_{i}\mid i\in I\}$
be the set of the simple components of $S$. We identify $\bar{A}$
with $S$. By Lemma \ref{lem:natural bi-module}, the $S$-bimodule
$R$ is a direct sum of copies of the natural left $S_{i}$-module
$U_{i0}$, the natural right $S_{j}$-module $U_{0j}$ and the natural
$S_{i}$-$S_{j}$- bimodule $U_{ij}$ for all $i,j\in I$. Note that
the $S$-bimodule $R$ has no components isomorphic to the trivial
$1$-dimensional $S$-bimodule $U_{00}$ as $A$ is $1$-perfect with
$R^{2}=0$. 

The proof is by induction on the length $\ell(R)$ of the $S$-bimodule
$R$. If $\ell(R)=1$, then $R=U_{ij}$ for some $i$ and $j$. Note
that $(i,j)\neq(0,0)$. Let $A_{1}=(S_{i}+S_{j})\oplus R$ and let
$A_{2}$ be the complement of $S_{i}+S_{j}$ in $S$. Then $A_{1}$
and $A_{2}$ are $1$-perfect. Note that $A_{2}A_{1}=A_{1}A_{2}=0$
so both $A_{1}$ and $A_{2}$ are ideals of $A$ with $A=A_{1}\oplus A_{2}$.
Hence $L=L_{1}\oplus L_{2}$, where $L_{i}=[A_{i},A_{i}]$ for $i=1,2$.
Since $L$ satisfies the conditions of Lemma \ref{lem:B=00003DB1+B2 if L=00003DL1+L2},
we have $B=B_{1}\oplus B_{2}$, where $B_{i}$ is a $\bar{B}_{i}$-minimal
Jordan-Lie inner ideal of $L_{i}$, $i=1,2$. Since $A_{2}$ is semisimple,
$B_{2}$ splits in $A_{2}$. Note that $B_{1}$ satisfies the conditions
of one of the Propositions \ref{prop:RA=00003D0}, \ref{prop:AR=00003D0},
\ref{prop: R is S-bimodule} and \ref{prop: A=00003DS+R, S=00003DS1+S2},
so $B_{1}$ splits in $A_{1}$. Therefore, $B$ splits in $A$. 

Assume that $\ell(R)>1$. Consider any maximal $S$-submodule $T$
of $R$, so $\ell(T)=\ell(R)-1$. Then $T$ is an ideal of $A$. Let
$\tilde{A}=A/T$. Denote by $\tilde{B}$ and $\tilde{R}$ the images
of $B$ and $R$ in $\tilde{A}$. Since $\ell(\tilde{R})=1$, by the
base of induction, $\tilde{B}$ splits, so there is a Levi subalgebra
$S'\cong S$ of $\tilde{A}$ such that $\tilde{B}=\tilde{B}_{S'}\oplus\tilde{B}_{R}$,
where $\tilde{B}_{S'}=\tilde{B}\cap S'$ and $\tilde{B}_{R}=\tilde{B}\cap\tilde{R}$.
Let $P$ be the full preimage of $\tilde{B}_{S'}$ in $B$. Then $\tilde{P}=\tilde{B}_{S'}\subseteq S'$,
so $P$ is a subspace of $B$ with $\bar{P}=\bar{B}$. Let $G$ be
the full preimage of $S'$ in $A$. Then $G$ is clearly $1$-perfect
(i.e. $G=\ccP_{1}(G)$), $\rad(G)=T$, $G/T\cong S$ and $P\subseteq B\cap G$.
Put $P_{1}=[P,[P,S'^{(1)}]]\subseteq G^{(1)}$. Then $P_{1}\subseteq[B,[B,A^{(1)}]]\subseteq B$,
so $P_{1}\subseteq B\cap G^{(1)}$. Note that $B'=B\cap G^{(1)}$
is a Jordan-Lie inner ideal of $G^{(1)}$ (because $G^{(1)}$ is a
subalgebra of $A^{(1)}$). Since 
\[
\bar{P}_{1}=[\bar{P},[\bar{P},\bar{S}'^{(1)}]]=[\bar{B},[\bar{B},\bar{A}^{(1)}]]=\bar{B},
\]
we get that $\bar{B}=\bar{P}_{1}\subseteq\bar{B}'\subseteq\bar{B}$,
so $\bar{B}'=\bar{B}$. Note that $G$ is a large subalgebra of $A$
(see Definition \ref{def:large }). Let $B''\subseteq B'$ be a $\bar{B}'$-minimal
Jordan-Lie inner ideal of $G^{(1)}$. As $G$ is $1$-perfect and
$\ell(T)<\ell(R)$, by the inductive hypothesis, $B''$ splits in
$G$. Since $B''\subseteq B'\subseteq B$ and $\bar{B}''=\bar{B}'=\bar{B}$,
by Proposition \ref{lem:Large is splits}, $B$ splits in $A$. 
\end{proof}
The following result follows from Theorem \ref{thm:minimal splits if R^2=00003D0}
and Proposition \ref{prop:C sub B splits in A}. 
\begin{cor}
\label{cor:B splits if R^2=00003D0} Let $L=[A,A]$ and let $B$ be
a Jordan-Lie inner ideal of $L$. Suppose that $p\neq2,3$, $A$ is
$1$-perfect, and $R^{2}=0$. Then $B$ splits in $A$. 
\end{cor}

Now, we are ready to prove Theorem \ref{thm:B=00003DeAf }. 
\begin{proof}[Proof of Theorem \ref{thm:B=00003DeAf }]
 (i) Recall that $B$ is bar-minimal. Since $R=\rad A$ is nilpotent,
there is an integer $m$ such that $R^{m-1}\neq0$ and $R^{m}=0$.
The proof is by induction on $m$. If $m=2$ , then by Theorem \ref{thm:minimal splits if R^2=00003D0},
$B$ splits. Suppose that $m>2$. Put $T=R^{2}\neq0$ and consider
$\tilde{A}=A/T$. Let $\tilde{B}$ and $\tilde{R}$ be the images
of $B$ and $R$ in $\tilde{A}$. Then we have $\tilde{R}=\rad\tilde{A}$,
$\tilde{R}^{2}=0$ and $\tilde{A}$ satisfies the conditions of the
Corollary \ref{cor:B splits if R^2=00003D0}. Hence, there is a Levi
subalgebra $S'$ of $\tilde{A}$ such that $\tilde{B}=\tilde{B}_{S'}\oplus\tilde{B}_{R}$,
where $\tilde{B}_{S'}=\tilde{B}\cap S'$ and $\tilde{B}_{R}=\tilde{B}\cap\tilde{R}$.
Let $P$ be the full preimage of $\tilde{B}_{S'}$ in $B$. Then $\tilde{P}=\tilde{B}_{S'}\subseteq S'$,
so $P$ is a subspace of $B$ with $\bar{P}=\bar{B}$. Let $G$ be
the full preimage of $S'$ in $A$. Then $G$ is a large subalgebra
of $A$ with $P\subseteq G\cap B$. Put $P_{1}=[P,[P,S'^{(1)}]]$
and $B_{1}=B\cap G^{(1)}$. Then $P_{1}\subseteq[B,[B,A^{(1)}]]\subseteq B$
and $P_{1}\subseteq[G,[G,G]]\subseteq G^{(1)}$, so $P_{1}\subseteq B\cap G^{(1)}=B_{1}$.
Since 
\[
\bar{P}_{1}=[\bar{P},[\bar{P},\bar{S}'^{(1)}]]=[\bar{B},[\bar{B},\bar{A}^{(1)}]]=\bar{B},
\]
we get that $\bar{B}=\bar{P}_{1}\subseteq\bar{B}_{1}\subseteq\bar{B}$,
so $\bar{B}_{1}=\bar{B}$. As $G^{(1)}$ is a Lie subalgebra of $A^{(1)}$,
$B_{1}=B\cap G^{(1)}$ is a Jordan-Lie inner ideal of $G^{(1)}$.
Put $B_{2}=\core_{G^{(1)}}(B_{1})$. Then by Proposition \ref{prop:Cbar},
$B_{2}$ is a Jordan-Lie inner ideal of $\ccP_{1}(G)^{(1)}$ such
that $B_{2}\subseteq B$ and $\bar{B}_{2}=\bar{B}$. Let $B_{3}\subseteq B_{2}$
be any $\bar{B}_{2}$-minimal inner ideal of $\ccP_{1}(G)^{(1)}$.
Since $\ccP_{1}(G)$ is $1$-perfect and $\rad(\ccP_{1}(G))^{m-1}\subseteq T^{m-1}=R^{2(m-1)}=0$,
by the inductive hypothesis, $B_{3}$ splits in $\ccP_{1}(G)$. Since
$\bar{B}_{3}=\bar{B}_{2}=\bar{B}$, by Lemma \ref{lem:Large is splits},
$B$ splits in $A$. 

(ii) We wish to show that $B=eAf$ for some strict orthogonal idempotent
pair $(e,f)$ in $A$. By (i), there is a $B$-splitting Levi subalgebra
$S$ of $A$ such that $B=B_{S}\oplus B_{R}$, where $B_{S}=B\cap S$
and $B_{R}=B\cap R$. Let $\{S_{i}\mid i\in I\}$ be the set of the
simple components of $S$, so $S=\bigoplus_{i\in I}S_{i}$. We identify
$\bar{A}$ with $S$. By Lemma \ref{lem:B=00003DB1+...}, we have
$\bar{B}=\bigoplus_{i\in I}X_{i}$, where $X_{i}=\bar{B}\cap S_{i}$
for all $i\in I$. Put $J=\{i\in I\mid X_{i}\ne0\}$. By Lemma \ref{lem:X=00003Dspan=00007Be:1<i<k<l<j<n=00007D},
for each $r\in J$ there is a matrix realization $M_{n_{r}}(\bbF)$
of $S_{r}$ and integers $1\leq k_{r}<l_{r}\leq n_{r}$ such that
$X_{r}$ is spanned by the set 
\[
E_{r}=\{e_{st}^{r}\mid1\leq s\leq k_{r}<l_{r}\leq t\leq n_{r}\}\subseteq S_{r}
\]
where $\{e_{ij}^{r}\mid1\leq i,j\leq n_{r}\}$ is a basis of $S_{r}$
consisting of matrix units. Let $e=\sum_{r\in J}\sum_{i=1}^{k_{r}}e_{ii}^{r}$
and $f=\sum_{r\in J}\sum_{j=l_{r}}^{n_{r}}e_{jj}^{r}$. Then $(e,f)$
is a strict orthogonal idempotent pair in $A$ with $B_{S}=\bigoplus_{i\in J}X_{i}=eSf$.
Note that $eAf$ is a Jordan-Lie inner ideal of $[A,A]$ with $\overline{eAf}=eSf=\bar{B}$.
We are going to show that $eRf\subseteq B_{R}$. This will imply $eAf=B$
as $B$ is bar-minimal. 

By Lemma \ref{lem:natural bi-module}, the $S$-bimodule $R$ is a
direct sum of copies of the natural left $S_{i}$-module $U_{i0}$,
the natural right $S_{j}$-module $U_{0j}$, the natural $S_{i}$-$S_{j}$-
bimodule $U_{ij}$ and the trivial $1$-dimensional $S$-bimodule
$U_{00}$ for all $i,j\in I$. Let $M$ be any minimal $S$-submodule
of $R$. It is enough to show that $eMf\subseteq B$. Fix $r,q\in I$
such that $M\cong U_{rq}$. We can assume that $r,q\in J$ (otherwise
$eMf=\{0\}\subseteq B$). Let $\{f_{ij}^{rq}\mid1\leq i\leq n_{r},\ 1\leq j\leq n_{q}\}$
be the standard basis of $M$ consisting of matrix units, such that
the action of $S_{r}$-$S_{q}$ on $M$ corresponds to matrix multiplication.
Note that 
\[
eMf=span\{f_{st}^{rq}\mid1\leq s\leq k_{r},\ l_{q}\leq t\leq n_{q}\}.
\]
We need to show that $f_{st}^{rq}\in B$ for all $s$ and $t$. First,
consider the case when $r=q$. Then $s\le k_{r}<l_{r}\le t$, so $s\ne t$.
Since $e_{st}^{r}\in B$ and $f_{ts}^{rr}=[e_{tt}^{r},f_{ts}^{rr}]\in L$,
by Lemma \ref{lem:Bar=000026Row 4.1}, we have 
\[
e_{st}^{r}f_{ts}^{rr}e_{st}^{r}=f_{ss}^{rr}e_{st}^{r}=f_{st}^{rr}\in B,
\]
as required. Assume now $r\neq q$. Fix any $e_{sj}^{r}\in E_{r}$
and $e_{it}^{q}\in E_{q}$. Since $e_{sj}^{r},e_{it}^{q}\in B$ and
$f_{ji}^{rq}=[e_{jj}^{r},f_{ji}^{rq}]\in L$, using Lemma \ref{lem:Jordan-Lie },
we obtain 
\[
\{e_{sj}^{r},f_{ji}^{rq},e_{it}^{q}\}=e_{sj}^{r}f_{ji}^{rq}e_{it}^{q}+e_{it}^{q}f_{ji}^{rq}e_{sj}^{r}=f_{st}^{rq}+0\in B,
\]
as required. 

(iii) Since $B=eAf$, by Lemma \ref{eAf is regular}, $B$ is regular. 
\end{proof}
\begin{cor}
\label{thm:B splits} Let $L=[A,A]$ and let $B$ be a Jordan-Lie
inner ideal of $L$. Suppose that $p\neq2,3$ and $A$ is $1$-perfect.
Then $B$ splits in $A$. 
\end{cor}

\begin{proof}
Let $B'\subseteq B$ be a bar-minimal Jordan-Lie inner ideal of $L$.
Then by Theorem \ref{thm:B=00003DeAf }(i), $B'$ splits in $A$.
Therefore, by Lemma \ref{prop:C sub B splits in A}, $B$ splits in
$A$. 
\end{proof}
Now we are ready to prove the main results of this paper. 
\begin{proof}[Proof of Theorem \ref{thm:main iff}]
 Suppose first that $B$ is bar-minimal. We need to show that $B=eAf$
for some strict orthogonal idempotent pair $(e,f)$ in $A$. By Lemma
\ref{lem: B =00003D =00005BB,=00005BB,L=00005D=00005D }(ii), $B$
is $L$-perfect, so by Lemma \ref{lem:B sub P(A)}, $B\subseteq\ccP_{1}(A)$
and $B$ is a Jordan-Lie inner ideal of $L_{1}=\ccP_{1}(A)^{(1)}$.
Let $C\subseteq B$ be a $\bar{B}$-minimal Jordan-Lie inner ideal
of $L_{1}$. Since $\ccP_{1}(A)$ is $1$-perfect, by Theorem \ref{thm:B=00003DeAf },
there exists a strict orthogonal idempotent pair $(e,f)$ in $\ccP_{1}(A)$
such that $C=e\ccP_{1}(A)f$. As $\ccP_{1}(A)$ is an ideal of $A$,
\[
CAC=e\ccP_{1}(A)fAe\ccP_{1}(A)f\subseteq e\ccP_{1}(A)f=C
\]
Hence, by Lemma \ref{lem:Bar=000026Row 4.1}(iii), $C$ is an inner
ideal of $L$ with $C\subseteq B$ and $\bar{C}=\bar{B}$. Since $B$
is bar-minimal, $C=B$. As $e,f\in\ccP_{1}(A)$, we have

\[
e\ccP_{1}(A)f\subseteq eAf=eeAf\subseteq e\ccP_{1}(A)Af\subseteq e\ccP_{1}(A)f.
\]
Therefore, $e\ccP_{1}(A)f=eAf$ and $B=C=eAf$ as required. 

Suppose now that $B=eAf$, where $(e,f)$ is a strict orthogonal idempotent
pair in $A$. We need to show that $B$ is bar-minimal. Let $C\subseteq B$
be a $\bar{B}$-minimal Jordan-Lie inner ideal of $L$. Then by the
``if'' part $C=e_{1}Af_{1}$ for some strict orthogonal idempotent
pair $(e_{1},f_{1})$ in $A$, so $e_{1}Af_{1}\subseteq eAf$ and
$\bar{e}_{1}\bar{A}\bar{f}_{1}=\bar{B}=\bar{e}\bar{A}\bar{f}$. Then
by Theorem \ref{thm:eAf sub e'Af' imply ...}(iv), there is a strict
idempotent pair $(e_{2},f_{2})$ in $A$ such that $(e_{2},f_{2})\le(e,f)$,
that is, $ee_{2}=e_{2}e=e_{2}$ and $f_{2}f=ff_{2}=f_{2}$. Moreover,
by Theorem \ref{thm:eAf sub e'Af' imply ...}(iv), $e_{2}Af_{2}=e_{1}Af_{1}=C$,
so $\bar{e}_{2}\bar{A}\bar{f}_{2}=\bar{B}=\bar{e}\bar{A}\bar{f}$.
We are going to show that $e_{2}=e$ (the proof of $f_{2}=f$ is similar).
Since $(e,f)$ is strict, by Theorem \ref{thm:eAf sub e'Af' imply ...}(iii)
, $\bar{e}_{2}\overset{\ccL}{\sim}\bar{e}$, so $\bar{e}=\bar{e}_{2}\bar{e}=\overline{e_{2}e}=\bar{e}_{2}$.
Hence, there is $r\in R$ such that $e_{2}=e+r$. We have 
\[
e+r=e_{2}=ee_{2}=e(e+r)=e+er,
\]
so $er=r$. Similarly, $re=r$. Since $e_{2}$ is an idempotent, 
\[
e+r=e_{2}=e_{2}^{2}=(e+r)^{2}=e+2r+r^{2}.
\]
Therefore, $r^{2}=-r$ and $r^{2^{k}}=-r$ for all $k\in\mathbb{N}$.
As $R$ is nilpotent, we get $r=0$, so $e_{2}=e$. Similarly, $f_{2}=f$.
Therefore, $B=eAf=e_{2}Af_{2}=C$, as required. 
\end{proof}

\begin{proof}[Proof of Corollary \ref{cor:innerposet}]
 First we claim that $\xi$ is well-defined, i.e. the image of $\xi$
consists of bar-minimal Jordan-Lie inner ideals of $L$. Let $(e,f)$
be a strict orthogonal idempotent pair in $A$. By Lemma \ref{lem:eAf cap Z =00003D 0}(iii),
$eAf$ is a Jordan-Lie inner ideal of $A^{(-)}$. Hence by Theorem
\ref{thm:main iff}, $eAf$ is bar-minimal, so $L$-perfect by Lemma
\ref{lem: B =00003D =00005BB,=00005BB,L=00005D=00005D }(ii). Now
by Lemma \ref{lem:B sub L^(infty)}, $eAf$ is a Jordan-Lie inner
ideal of $A^{(k)}$ for all $k\ge0$. Moreover, it is bar-minimal
by Theorem \ref{thm:main iff}, as required. By Theorem \ref{thm:main iff},
the map $\xi$ is surjective. The rest follows from Theorem \ref{thm:eAf sub e'Af' imply ...}(iv).
\end{proof}
\begin{proof}[Proof of Corollary \ref{cor:main regular}]
 Since $B$ is bar-minimal, by Theorem \ref{thm:main iff}, there
exists a strict orthogonal idempotent pair $(e,f)$ in $A$ such that
$B=eAf$. Therefore, by Lemma \ref{eAf is regular}, $B$ is regular.
\end{proof}

\begin{proof}[Proof of Corollary \ref{cor:main split}]
 Let $C\subseteq B$ be a $\bar{B}$-minimal Jordan-Lie inner ideal
of $L$. Then by Theorem \ref{thm:main iff}, there exists a strict
orthogonal idempotent pair $(e,f)$ in $A$ such that $C=eAf$, so
by Lemma \ref{lem:eAf splits}(i), $C$ splits in $A$. Therefore,
by Proposition \ref{prop:C sub B splits in A}, $B$ splits in $A$.
Moreover, $B=eSf\oplus B_{R}$ with $eRf\subseteq B_{R}=B\cap R$.
\end{proof}

\end{document}